\documentclass{svmult-turin}

\usepackage{type1cm} 

\usepackage{makeidx}         
\usepackage{graphicx}        

\usepackage[bottom]{footmisc}

\usepackage{newtxtext}      %
\usepackage{newtxmath}      

\usepackage{mathtools}
\usepackage[all]{xy}

\usepackage{subeqnarray}
\usepackage{amsmath}
\usepackage{amssymb}
\usepackage{enumitem}
\usepackage{nccmath}

\def\DBbR{\mathbb{R}}

\def\DBcF{\mathcal{F}}
\def\DBcS{\mathcal{S}}
\def\DBS0{M^1}
\def\DBcB{\mathcal{B}}
\def\DBrd{\DBbR^d}

\def\DBrdd{\DBbR^{2d}}

\newcommand{\DBGLL}{\mathrm{GL}\left(2d,\mathbb{R}\right)}

\begin{document}

\title*{Linear perturbations of the Wigner transform and the Weyl quantization}
\author{Dominik Bayer, Elena Cordero, Karlheinz Gr\"ochenig, and S. Ivan Trapasso}

\institute{Dominik Bayer \at Universit\"at der Bundeswehr, M\"unchen
  \\ Werner Heisenberg Weg 39 \\ D-85577 Neubiberg, Germany \\ \email{dominik.bayer@unibw.de}
\and Elena Cordero \at Universit\`a di Torino, Dipartimento di
Matematica, \\ via Carlo Alberto 10, 10123 Torino, Italy \\ \email{elena.cordero@unito.it}
\and Karlheinz Gr\"ochenig  \at Faculty of Mathematics,
University of Vienna, \\ Oskar-Morgenstern-Platz 1,  A-1090 Vienna, Austria \\
\email{karlheinz.groechenig@univie.ac.at}
\and Salvatore Ivan Trapasso \at Dipartimento di Scienze Matematiche ``G. L. Lagrange", Politecnico di Torino \\ Corso Duca degli Abruzzi 24, 10129 Torino, Italy \\ 
\email{salvatore.trapasso@polito.it}}

\maketitle
\abstract{We study a class of  quadratic time-frequency
  representations that, roughly speaking, are obtained by linear
  perturbations of the Wigner transform. They satisfy Moyal's formula
  by default and share many other properties with the Wigner
  transform, but in general they do not belong to   Cohen's class. We
  provide a characterization of the intersection of the two
  classes. To any such time-frequency representation, we associate a
  pseudodifferential calculus. We investigate  the related   quantization
  procedure, study the properties of the  pseudodifferential
  operators, and compare the formalism  with that of the  Weyl calculus.}  

\vfill


\pagebreak
\section{Introduction} The Wigner transform is a key concept  lying
at the heart of pseudodifferential operator theory and  time-frequency
analysis. It was  introduced by Wigner \cite{DBwigner}  as a
quasi-probability distribution in order to extend 
the  phase-space formalism of classical statistical mechanics
to the domain of quantum physics. Subsequently, this line of thought
led to the  phase-space formulation of quantum mechanics. In engineering,
the Wigner transform was considered the ideal tool for the
simultaneous  investigation of  temporal  and spectral features of
signals, because it enjoys all properties desired from a good 
time-frequency representation (except for positivity)~\cite{DBMH97}.

To be
precise, the (cross-)Wigner distribution of two signals $f,g \in
L^2(\mathbb{R}^d)$ is defined as  
\begin{equation}
W(f,g)(x,\omega )=\int_{\mathbb{R}^{d}}e^{-2\pi iy \cdot \omega }f\left(x+\frac{y}{2}\right)%
\overline{g\left(x-\frac{y}{2}\right)}\, dy \, .
\label{DBwig}
\end{equation}
If $f=g$ we write $Wf(x,\omega)$. This is an example of a quadratic
time-frequency representation, and  heuristically $Wf(x,\omega)$ 
is interpreted as  a measure of the energy content of the signal $f$
in a ``tight'' spectral band around $\omega$ during a
``short'' time interval near $x$.


The Wigner transform plays a  key role in the  Weyl quantization and
the corresponding pseudodifferential calculus. Quantization is a
formalism that associates a function on phase space (an
observable) with an operator on a Hilbert space.  The standard quantization rule in physics
is the Weyl correspondence: given the phase-space observable
$\sigma\in \DBcS'(\mathbb{R}^{2d})$ (called \emph{symbol} in
mathematical language) the corresponding Weyl transform $\mathrm{op_W}
(\sigma):\DBcS(\mathbb{R}^d)\rightarrow \DBcS'(\mathbb{R}^{d})$ is
(formally) defined by 
\begin{equation}
\mathrm{op_W} (\sigma)f(x)=\int_{\mathbb{R}^{2d}} e^{2\pi i (x-y)\cdot \omega} \sigma\left(\frac{x+y}{2},\omega \right)f(y) dy d\omega .
\label{DBweyltext}
\end{equation}

A formal computation reveals the role of the Wigner transform in this
definition, because 
\begin{equation}
\langle \mathrm{op_W} (\sigma)f,g\rangle =
\langle
\sigma,W(g,f)\rangle, \quad f,g\in \DBcS(\mathbb{R}^d).
\label{DBweyltdual}
\end{equation}
Whereas the rigorous interpretation of \eqref{DBweyltext} is subtle and
requires oscillatory integrals, the weak formulation \eqref{DBweyltdual}
is easy to handle and works without problems for distributional
symbols $\sigma $. The bracket  $\langle  f,g\rangle $ denote the
extension to $\DBcS' (\mathbb{R}^{d})\times \DBcS (\mathbb{R}^{d})$ of the
inner product on $L^2(\mathbb{R}^{d})$.

Unfortunately, not all properties which are desired from a time-frequency
representation are compatible. The
Wigner transform is real-valued, but 
it may take negative values. This is a serious obstruction to the
interpretation of the Wigner transform  as a  probability distribution
or as an energy density of  a signal.  By  Hudson's
theorem~\cite{DBhudson 74}
only generalized Gaussian functions have  positive Wigner transforms. 

To obtain time-frequency representations that are positive for all
functions, one takes local averages of the Wigner transform  in order
to tame the sign oscillations. This is usually done  by  convolving
$Wf$ with a  suitable kernel $\theta $.  This idea 
yields a class of  quadratic time-frequency
representations, which is called  Cohen's class~\cite{DBcohen gdist 66}. The
time-frequency representations in  Cohen's class are parametrized by
a  kernel $\theta \in \DBcS'(\mathbb{R}^{2d})$, and the associated
representation is defined as 
\begin{equation}\label{DBIntroCohendistr}
Q_{\theta}(f,g)\coloneqq W(f,g)*\theta,\qquad f,g\in \DBcS(\mathbb{R}^d).
\end{equation}
Most time-frequency representations proposed so far belong to Cohen's
class~\cite{DBcohen tfa 95, DBhlaw qtf}, and  the  correspondence between properties of $\theta
$ and $Q_{\theta}$  is well
understood~\cite{DBhlaw qtf,DBjanssen94c}. In many examples $Q_\theta $
can be interpreted as a perturbation of the Wigner distribution, while
retaining some of its important properties. 
Thus Cohen's  class provides a unifying  framework for the
study of time-frequency representations  appearing in signal
processing. See for instance \cite{DBcohen tfa 95, DBcohen weyl 12, DBhlaw book, DBhlaw qtf}.  


Next, for every time-frequency representation in Cohen's class one can
introduce  a  quantization rule in analogy to the Weyl quantization
\eqref{DBweyltdual}, namely, 
\begin{equation}
\langle \mathrm{op_{\theta}} (\sigma)f,g\rangle =
\langle
\sigma,Q_{\theta}(g,f)\rangle = \langle \sigma \ast \theta ^*, W(g,f)
\rangle, \quad f,g\in \DBcS(\mathbb{R}^{d}) \, ,
\label{DBcohenopdual}
\end{equation}
whenever the expressions make sense. Although the new
operator  $\mathrm{op}_\theta (\sigma )$ is just a Weyl operator with
the modified symbol $\sigma \ast \theta ^*$ (whenever defined in $\DBcS
'(\mathbb{R}^d )$), the variety of pseudodifferential calculi given 
by definition~\eqref{DBcohenopdual}  adds  flexibility and a new flavor    to the
description and analysis of operators.  

For example, a first variation of the Wigner transform are the $\tau
$-Wigner transforms
\begin{equation}
W_{\tau }(f,g)(x,\omega )=\int_{\mathbb{R}^{d}}e^{-2\pi iy\cdot \omega }f(x+\tau y)%
\overline{g(x-(1-\tau )y)}\,dy,\quad f,g\in \DBcS(\mathbb{R}^{d}) \, .
\label{DBtauwig}
\end{equation}
These belong to Cohen's class and possess the kernel $\theta _\tau \in
\DBcS (\mathbb{R}^{2d})$ with Fourier transform 
\begin{equation}\label{DBkerneltau}
\widehat{\theta_{\tau}}\left(\xi,\eta\right)= e^{-2\pi i \left(\tau -
	\frac{1}{2}\right)\xi \cdot \eta}, \quad
(\xi,\eta)\in\mathbb{R}^{2d} \, . 
\end{equation}
The corresponding pseudodifferential calculi are the  Shubin $\tau$-pseudodifferential
operators $\mathrm{op_{\theta_\tau}} (\sigma)$ in formula
\eqref{DBcohenopdual}. For the parameter $\tau =1/2$ this is the Weyl
calculus, for $\tau  =0$ this is the Kohn-Nirenberg calculus. The
important  Born-Jordan quantization rule is obtained as an average
over $\tau \in [0,1] $, see \cite{DBbogetal} or the textbook \cite{DBdg bj}. 

From  \eqref{DBkerneltau} we see that the deviation from the Wigner
transform is measured by  $\mu= \tau-1/2$. Hence a natural
generalization of the $\tau$-Wigner transform is the  replacement of the scalar parameter $\mu$ with a matrix expression $M=T-(1/2)I$, with $T\in \mathbb{R}^{d\times d}$. This  gives
the family of matrix-Wigner distributions  
\begin{equation}\label{DBWM}
W_M (f,g)(x,\omega)=\int_{\mathbb{R}^{d}}e^{-2\pi iy \cdot \omega
	}f\left(x+\left(M+\frac{1}{2}I\right)y\right)\overline{g\left(x+\left(M-\frac{1}{2}I\right)
	y\right)} dy \, .
\end{equation}
Again these  are members of the Cohen class in \eqref{DBIntroCohendistr}
with a kernel $\theta _M$ given by its Fourier transform 
\begin{equation}\label{DBkernelM}
\widehat{\theta_{M}}(\xi,\eta)=e^{-2\pi i \xi \cdot M\eta}.
\end{equation}

An even  more general definition in the spirit of \eqref{DBWM}
uses an arbitrary linear mapping of the pair $(x,y) \in
\mathbb{R}^{2d}$. Let $A=\left(\begin{array}{cc}
A_{11} & A_{12}\\
A_{21} & A_{22}
\end{array}\right)$ be an  invertible, real-valued $2d\times
2d$-matrix. We define the matrix-Wigner transform $\mathcal{B}_A$ of two 
functions $f,g$ by 
\begin{equation}\label{DBBAe}
\mathcal{B}_{A}\left(f,g\right)\left(x,\omega\right)=\int_{\mathbb{R}^{d}}e^{-2\pi
	iy \cdot \omega}f\left(A_{11}x+A_{12}y\right)\overline{g\left(A_{21}x+A_{22}y\right)}dy
\, .
\end{equation}
Clearly,  $W_M$ in \eqref{DBWM} is a special case  by choosing 
\begin{equation}\label{DBAM}
A=A_{M}=\left(\begin{array}{cc}
I\quad & M+(1/2)I\\
I\quad & M-(1/2)I
\end{array}\right) \, .
\end{equation}

Once again, every matrix Wigner transform $\mathcal{B}_A$ is
associated with a pseudodifferential calculus or a quantization rule. Given
an invertible $2d\times 2d$ -matrix $A$ and a symbol $\sigma \in
\mathcal{S}' (\mathbb{R}^{2d}) $, we define the 
operator $\sigma ^A$ by 
\begin{equation}
\label{DBeq:c1}
\left\langle \sigma^{A}f,g\right\rangle \equiv \left\langle
\sigma,\mathcal{B}_{A}\left(g,f\right)\right\rangle ,\qquad
f,g\in \DBcS (\DBrd)  \, .
\end{equation}
This is then a continuous operator from $\DBcS (\DBrd ) $ to $\DBcS
'(\DBrd )$ and the mapping $\sigma \to \sigma ^A$ is a form of
quantization similar to the Weyl quantization.

The class of matrix Wigner transforms has already a sizeable
history. To our knowledge they were first introduced in \cite{DBfeig micchelli} in dimension $d=1$ for a different purpose, but the
original  contribution to the subject went by  unnoticed. The first
thorough investigation of the
matrix Wigner  transforms $\mathcal{B}_A$ is contained in the
unpublished  Ph.\  D.\ thesis \cite{DBbayer} of the first-named author
who studied the general properties of this class of time-frequency
representations and the associated pseudodifferential
operators. Independently, \cite{DBbco quadratic} introduced and studied these
``Wigner representations associated with linear transformations of the
time-frequency plane'' in dimension $d=1$. In~\cite{DBgoh} matrix
Wigner transforms were used for a signal estimation problem.
Recently, in~\cite{DBtoft bil 17} Toft discusses ``matrix parametrized
pseudo-differential calculi on modulation spaces'', which correspond
to the time-frequency representations in \eqref{DBWM}. Finally, two of
us~\cite{DBCT18} took up and reworked  and streamlined several results of
\cite{DBbayer} and determined the precise intersection between matrix Wigner
transforms and  Cohen's class. 

\vspace{3mm}

The goal of this chapter is a systematic survey  of the accumulated
knowledge about the matrix Wigner transforms and their
pseudodifferential calculi.

In the first part (Section~3) we discuss the general properties of
the matrix Wigner transforms.

(i) We state the main formulas for covariance, the behavior with
respect to the Fourier transform, the analog of   Moyal's formulas,
and the inversion formula.

(ii) It is well-known that, up to normalization,  the ambiguity
function and the short-time   Fourier transform are just different
versions and names for the  Wigner transform. This is no longer true
for the matrix Wigner transforms, so we give precise conditions on the
parametrizing matrix $A$ so that $\mathcal{B}_A$ can be expressed as
a short-time Fourier transform, up to  a
phase factor and a change of coordinates. 

(iii) Of special importance is the intersection of the class of
matrix Wigner transforms with Cohen's class. After a partical result
in \cite{DBbayer}, it was proved in~\cite{DBCT18} that $\mathcal{B}_A$
belongs to Cohen's class, if and only if $
A= \Big(
\begin{smallmatrix}
I\, & M+ I/2\\
I\, & M-I/2
\end{smallmatrix}\Big)
$ for some $d\times d$-matrix $M$. Thus in general a matrix Wigner
transform does not belong to Cohen's class. This fact explains why for
certain results we  have to impose assumptions on the parametrizing
matrix $A$. 

(iv) A further item is  the boundedness of the bilinear  mapping
$(f,g) \to \mathcal{B}_A(f,g)$  on various function spaces. These
results are quite useful in the analysis of the mapping properties of
the pseudodifferential operators $\sigma ^A$. 

\vspace{3mm}

In the second part (Section~4)  we study the pseudodifferential
calculi defined by the rule  \eqref{DBeq:c1}. 

(i) Based on Feichtinger's kernel theorem (see Theorem \ref{DBfei ker thm}), we first show that every
``reasonable'' operator can be represented as a pseudodifferential
operator $\sigma ^A$. We remark that the map $(\sigma,A) \mapsto \sigma^{A}$ is highly non-injective and, given two matrices $A$ and $B$ and two symbols $\sigma,\rho$, we obtain formulas characterizing the condition $\sigma ^A = \rho ^B$. 

(ii) A large section is devoted to the mapping properties of the
pseudodifferential operator $\sigma ^A $ on various function spaces,
in particular on $L^p$-spaces and on modulation spaces.

(iii) Finally, we extend the boundedness results  for symbols in the
Sj\"ostrand class to those pseudodifferential operators for which
$\mathcal{B}_A$ is in Cohen's class. 

\vspace{3mm}

For most results we will include proofs, but we will omit those proofs
that only require a formal computation. We hope that the
self-contained and comprehensive presentation of matrix Wigner
distributions and their pseudodifferential calculi  will offer some
added value when compared with the focussed, individual publications.

\section{Preliminaries}
\textbf{Notation.} We define $t^2=t\cdot t$, for $t\in\mathbb{R}^d$, and
$x\cdot y$ is the scalar product on $\mathbb{R}^{d}$. The Schwartz class is denoted by  $\mathcal{S}(\mathbb{R}^{d})$, the space of temperate distributions by  $\mathcal{S}'(\mathbb{R}^{d})$.   The brackets  $\langle  f,g\rangle $ denote the extension to $\DBcS' (\mathbb{R}^{d})\times \DBcS (\mathbb{R}^{d})$ of the inner product $\langle f,g\rangle=\int f(t){\overline {g(t)}}dt$ on $L^2(\mathbb{R}^{d})$.  The conjugate exponent $p'$ of $p \in [1,\infty]$ is defined by $1/p+1/p'=1$. The symbol $\lesssim_{\lambda}$ means that the underlying inequality holds up to a positive constant factor $C=C(\lambda)>0$ that depends on the parameter $\lambda$:
$$ f\lesssim g\quad\Rightarrow\quad\exists C>0\,:\,f\le Cg. $$
If $ f\lesssim g$ and $g\lesssim f$ we write $f\asymp g$.\par 
The Fourier transform of a function $f\in \mathcal{S}(\mathbb{R}^{d})$ is normalized as
\[
\mathcal{F} f(\omega)= \int_{\mathbb{R}^{d}} e^{-2\pi i  x \cdot\omega} f(x)\, dx,\qquad \omega \in \mathbb{R}^{d}.
\]

For any $x,\omega \in \mathbb{R}^{d}$, the modulation $M_{\omega}$ and translation $T_{x}$ operators are defined as 
\[
M_{\omega}f\left(t\right)= e^{2\pi it \cdot \omega}f\left(t\right),\qquad T_{x}f\left(t\right)= f\left(t-x\right).
\]
Their composition $\pi(x,\omega)=M_\omega T_x$ is called a time-frequency shift.

Given a complex-valued function $f$ on $\mathbb{R}^{d}$, the involution $f^{*}$ is defined as $$f^*(t)\coloneqq \overline{f(-t)}, \quad t\in\mathbb{R}^{d}.$$ 

The short-time Fourier transform (STFT) of a signal $f\in \DBcS'(\mathbb{R}^{d})$ with respect to the window function $g \in \DBcS(\mathbb{R}^{d})$ is defined as
\begin{equation}\label{DBSTFTdef}
V_gf(x,\omega)=\langle f,\pi(x,\omega ) g\rangle=\mathcal{F} (f\cdot \overline{T_x g})(\omega)=\int_{\mathbb{R}^{d}}
f(y)\, {\overline {g(y-x)}} \, e^{-2\pi iy \cdot \omega }\, dy.
\end{equation}

The group of invertible, real-valued $2d\times 2d$ matrices is denoted by
\[
\mathrm{GL}\left(2d,\mathbb{R}\right)=\left\{ M\in\mathbb{R}^{2d\times2d}\,|\,\det M\ne0\right\},
\] 
and we denote the transpose of an inverse matrix by
\[
M^{\#}\equiv (M^{-1})^{\top} = (M^{\top})^{-1}, \qquad M\in \mathrm{GL}\left(2d,\mathbb{R}\right).
\]

Let $J$ denote the canonical symplectic matrix in $\mathbb{R}^{2d}$, namely
\[
J=\left(\begin{array}{cc}
0_{d} & I_{d}\\
-I_{d} & 0_{d}
\end{array}\right).
\]

Observe that, for $z=\left(z_{1},z_{2}\right)\in\mathbb{R}^{2d}$, we have
$Jz=J\left(z_{1},z_{2}\right)=\left(z_{2},-z_{1}\right),$  $J^{-1}z=J^{-1}\left(z_{1},z_{2}\right)=\left(-z_{2},z_{1}\right)=-Jz,$ and 
$J^{2}=-I_{2d\times2d}.$

\subsection{Function spaces} 
Recall that $C_{0}(\DBrd)$ denotes the class of continuous functions on $\DBrd$ vanishing at infinity.

\noindent
\textbf{Modulation spaces.} Fix a non-zero window $g\in\mathcal{S}(\DBrd)$ and $1\leq p,q\leq
\infty$. \\
$(i)$ The {\it	modulation space} $M^{p,q}(\DBrd)$ consists of all temperate
distributions $f\in\mathcal{S}'(\DBrd)$ such that $V_gf\in L^{p,q}(\DBrdd )$
(mixed-norm Lebesgue space). The norm on $M^{p,q}$ is
$$
\|f\|_{M^{p,q}}=\|V_gf\|_{L^{p,q}}=\left(\int_{\DBrd}
\left(\int_{\DBrd}|V_gf(x,\omega)|^p \,
dx\right)^{q/p}d\omega \right)^{1/q},
$$
(with obvious modifications for $p=\infty$ or $q=\infty$).
If $p=q$, we write $M^p$ instead of $M^{p,p}$. 

\noindent
 $(ii)$ The {\it modulation space} $W(\DBcF L^p,L^q)(\DBrd)$ can be defined as the space of distributions $f\in \mathcal{S}'(\DBrd)$ such that
\[
\|f\|_{W(\DBcF L^p,L^q)(\DBrd)}:=\left(\int_{\DBrd}
\left(\int_{\DBrd}|V_gf(x,\omega)|^p \,
d\omega\right)^{q/p} d x\right)^{1/q}<\infty  \,
\]
(with obvious modifications for $p=\infty$ or $q=\infty$).
It can also be viewed as a special case of Wiener amalgam spaces \cite{DBfeichtinger90}. Though their inventor H.G. Feichtinger nowadays suggests to call it  modulation  rather than Wiener amalgam space, since its definition involves the mixed-Lebesgue norm of the short-time Fourier transform of $f$.

The so-called fundamental identity of time-frequency analysis 
\begin{equation}\label{DBFI}
V_{g}f\left(x,\omega\right)=e^{-2\pi ix \cdot \omega}V_{\hat{g}}\hat{f}\left(\omega,-x\right)
\end{equation} implies that
$$ \| f \|_{{M}^{p,q}} = \left( \int_{\DBrd} \| \hat f \ T_{\omega} \overline{\hat g} \|_{\DBcF L^p}^q (\omega) \ d \omega \right)^{1/q}
= \| \hat f \|_{W(\DBcF L^p,L^q)}. $$
Hence the $W(\DBcF L^p,L^q)$ spaces under our consideration are simply the image via the Fourier transform of the  ${M}^{p,q}$ spaces:
\begin{equation}\label{DBW-M}
\DBcF ({M}^{p,q})=W(\DBcF L^p,L^q).
\end{equation}

The spaces $M^{p,q}(\DBrd)$ and $W(\DBcF L^p,L^q)(\DBrd)$ are  Banach spaces. Further, 
their definition is independent of the choice of the window $g$. The reader may find a comprehensive discussion of these function spaces in \cite{DBfeichtinger90, DBF1, DBGrochenig_2001_Foundations}. 
In particular, for $p=q$, we obtain
$$ M^p(\DBrd)=W(\DBcF L^p,L^p)(\DBrd),\quad 1\leq p\leq \infty.$$

Recall that the class of admissible windows for computing the modulation space norm can be extended to $M^1 \setminus\{0\}$ (cf. \cite[Thm.~11.3.7]{DBGrochenig_2001_Foundations} and \cite[Thm. 2.2]{DBcnt18}). \\

For $p=q=2$, we have $$M^2(\DBrd)=L^2(\DBrd),$$
the Hilbert space of square-integrable functions. 
%
 Among the properties of modulation spaces we list the following results.
 \begin{lemma}\label{DBmo} For $1\leq p,q,p_1,q_1,p_2,q_2\leq \infty$, we have\\
 	(i) $M^{p_1,q_1}(\DBrd)\hookrightarrow M^{p_2,q_2}(\DBrd)$, if
 	$p_1\leq p_2$ and $q_1\leq q_2$. \\
 	(ii) If $1\leq p,q<\infty$, then $(M^{p,q}(\DBrd))'=M^{p',q'}(\DBrd)$.\\
 \end{lemma}

\subsection{Basic Properties of  $M^1$}
Here we collect the main properties of the modulation space $M^1(\DBrd)$, also known as the Feichtinger algebra \cite{DBfei segal}.  
By Lemma \ref{DBmo} $(ii)$ we infer its dual space
$$(M^1(\DBrd))'=M^\infty(\DBrd)$$
and the inclusions $M^1(\DBrd)\hookleftarrow M^{p,q}(\DBrd)$, for every $1\leq p,q\leq \infty$. In particular,
$$ M^1(\DBrd) \hookrightarrow L^2(\DBrd) \hookrightarrow M^\infty(\DBrd).$$
There are plenty of equivalent characterizations for the Feichtinger algebra, we refer the interested reader to cf. \cite{DBjakob}.
\begin{proposition}\label{DBprop M1}
	\begin{enumerate}
		\item[(i)] $\left(M^1(\DBrd),\left\Vert \cdot\right\Vert _{M^1}\right)$ is a Banach space for any fixed non-zero window $g\in M^1(\DBrd)$, and different windows yield equivalent norms. 
		\item[(ii)] $\left(M^1(\DBrd),\left\Vert \cdot\right\Vert _{M^1}\right)$
		is a time-frequency homogeneous Banach space: for any $z\in\mathbb{R}^{2d}$, $f\in M^1(\DBrd)$,
		one has $\pi\left(z\right)f\in M^1(\DBrd)$
		and $\left\Vert \pi\left(z\right)f\right\Vert _{M^1}=\left\Vert f\right\Vert _{M^1}$.
		In particular, it is the smallest time-frequency homogeneous Banach
		space containing the Gaussian function. 
		\item[(iii)] The Schwartz class $\DBcS(\DBrd)$ is a
		subset of $M^1(\DBrd)$ and $L^{2}(\DBrd)$
		is the completion of $M^1(\DBrd)$ with respect
		to $\left\Vert \cdot\right\Vert _{L^{2}}$ norm. 
		\item[(iv)] $M^1(\DBrd)$ is invariant under the Fourier
		transform, i.e. for any $f\in M^1(\DBrd)$ one
		has $\mathcal{F}f\in M^1(\DBrd)$ and $\left\Vert \mathcal{F}f\right\Vert _{M^1}=\left\Vert f\right\Vert _{M^1}$. 
	\end{enumerate}
\end{proposition}

Recall that the tensor product of two functions $f,g:\mathbb{R}^{d}\rightarrow\mathbb{C}$ it
is defined as 
\[
f\otimes g:\mathbb{R}^{2d}\rightarrow\mathbb{C}\,:\,\left(x,y\right)\mapsto f\otimes g\left(x,y\right)= f\left(x\right)g\left(y\right).
\]
Clearly the tensor product  $\otimes$ is a bilinear bounded mapping from $L^{2}(\DBrd)\times L^{2}(\DBrd)$ into $L^{2}\left(\mathbb{R}^{2d}\right)$. Feichtinger algebra is well behaved under tensor products and we list a few results. 
\begin{proposition} \label{DBtensor prod S0} \begin{enumerate}[label=(\roman*)]
	\item The tensor product $$\otimes: M^1(\DBrd)\times M^1(\DBrd) \rightarrow M^1\left(\mathbb{R}^{2d}\right)$$ is a bilinear bounded operator. 
	\item $\DBS0$ enjoys the \textit{tensor factorization property}\footnote{That is, \[ M^1(\mathbb{R}^{2d})\simeq M^1 (\mathbb{R}^d) \widehat{\otimes} M^1 (\mathbb{R}^d),\] where the symbol $\widehat{\otimes}$ denotes the projective tensor product (see \cite{DBjakob} for the details).}: the space $\DBS0 (\mathbb{R}^{2d})$ consists of all functions of the form 
	\[ f = \sum_{n\in \mathbb{N}} g_{n}\otimes h_n, \]
where $\{g_n\},\{h_n\}$ are (sequences of) functions in $\DBS0(\DBrd)$ such that 
\[ \sum_{n\in \mathbb{N}} \left\Vert g_{n} \right\Vert_{\DBS0} \left\Vert h_n \right\Vert_{\DBS0} < \infty. \]
	\item The tensor product is well defined on $M^{\infty}$: for any $f,g\in M^{\infty}(\mathbb{R}^d)$, $f\otimes g$ is the unique element of $M^{\infty}(\mathbb{R}^{2d})$ such that
	\[
	\left\langle f\otimes g,\phi_{1}\otimes\phi_{2}\right\rangle \equiv\left\langle f,\phi_{1}\right\rangle \left\langle g,\phi_{2}\right\rangle, \quad \forall \phi_{1},\phi_{2}\in M^1(\mathbb{R}^d).
	\]
\end{enumerate}
\end{proposition}

Just as the (temperate) distributions are related to the Schwartz kernel theorem, there is an important kernel theorem in the context of time-frequency analysis, the \textit{Feichtinger kernel theorem} \cite{DBfei gro kernel}.

\begin{theorem} \label{DBfei ker thm} \begin{enumerate}[label=(\roman*)]
		\item Every distribution $k \in M^{\infty}(\mathbb{R}^{2d})$ defines a bounded linear operator $T_k : M^1(\mathbb{R}^d) \rightarrow M^{\infty}(\mathbb{R}^d)$ according to 
		\[ \langle T_k f,g \rangle = \langle k, g \otimes \overline{f} \rangle, \quad \forall f,g\in M^1(\mathbb{R}^d),\] 
		with $\left\Vert T_k \right\Vert_{M^1\rightarrow M^\infty} \leq \left\Vert k \right\Vert_{M^{\infty}}$.
		\item For any bounded operator $T : M^1(\mathbb{R}^d) \rightarrow M^{\infty}(\mathbb{R}^d)$ there exists a unique kernel $k_T \in M^{\infty}(\mathbb{R}^{2d})$ such that 
		\[ \langle T f,g \rangle = \langle k_T, g \otimes \overline{f} \rangle, \quad \forall f,g\in M^1(\mathbb{R}^d).\] 
		\end{enumerate}
\end{theorem}

The proofs of the aforementioned results can be found in the cited references. We wish to highlight the comprehensive survey \cite{DBjakob}, where the properties of the Feichtinger algebra are explored in full generality. 

\subsection{Bilinear coordinate transformations}
Let us summarize the properties of bilinear coordinate transformations in the time-frequency plane. Given a matrix $A=\left(\begin{array}{cc}
A_{11} & A_{12}\\
A_{21} & A_{22}
\end{array}\right) \in \mathbb{R}^{2d\times 2d}$ with $d\times d$ blocks $A_{ij}\in\mathbb{R}^{d\times d}$, $i,j=1,2$, we use the symbol $\mathfrak{T}_{A}$ to denote the transformation acting on a function $F:\DBrdd\rightarrow\mathbb{C}$ as
\begin{align*}
\mathfrak{T}_{A}F\left(x,y\right) & = F\left(A\left(\begin{array}{c}
x\\
y
\end{array}\right)\right) \\ & = F\left(A_{11}x+A_{12}y,A_{21}x+A_{22}y\right).
\end{align*}
The following lemma collects elementary facts on such transformations. 
\begin{lemma}
	\label{DBcoord trans isom}
	\begin{enumerate}[label=(\roman*)]
		\item For any $A,B \in\mathbb{R}^{2d\times2d}$ we have $\mathfrak{T}_{A}\mathfrak{T}_{B}=\mathfrak{T}_{BA}$.
		\item If $A\in\mathrm{GL}\left(2d,\mathbb{R}\right)$, the transformation
		$\mathfrak{T}_{A}$ is a topological isomorphism on $L^{2}(\DBrdd)$
		with inverse $\mathfrak{T}_{A}^{-1}=\mathfrak{T}_{A^{-1}}$ and adjoint
		$\mathfrak{T}_{A}^{*}=\left|\det A\right|^{-1}\mathfrak{T}_{A^{-1}}$. 
		\item If $A\in\mathrm{GL}\left(2d,\mathbb{R}\right)$, the transformation
		$\mathfrak{T}_{A}$ is an isomorphism on $M^1(\DBrdd)$. 
		\item For any $A\in\DBGLL$, $f\in L^2(\DBrd)$, $x,\omega \in \DBrd$: 
		\[
		\mathfrak{T}_{A}T_{x}f=T_{A^{-1}x}\mathfrak{T}_{A}f,\qquad\mathfrak{T}_{A}M_{\omega}f=M_{A^{\top}\omega}\mathfrak{T}_{A}f.
		\]
	\end{enumerate}
\end{lemma}

\begin{proof}
	The only non-trivial issue is the continuity on $M^1$. By \cite[Prop. 3.1]{DBcn met rep} we have 
	\[
	\left\Vert \mathfrak{T}_{A}F\right\Vert _{M^1}\le C\left|\det A\right|^{-1}\left(\det\left(I+A^{\top}A\right)\right)^{1/2}\left\Vert F\right\Vert _{M^1},
	\]
for some constant $C>0$.
\end{proof}
Two special transformations deserve a separate notation. The first is the \emph{flip} operator
\[
\tilde{F}\left(x,y\right)\equiv\mathfrak{T}_{\tilde{I}}F\left(x,y\right)=F\left(y,x\right),\qquad\tilde{I}=\left(\begin{array}{cc}
0_{d} & I_{d}\\
I_{d} & 0_{d}
\end{array}\right)\in\DBGLL,
\]
while the other one is the \emph{reflection} operator: 
\[
\mathcal{I}F\left(x,y\right)\equiv\mathfrak{T}_{-I}F\left(x,y\right)=F\left(-x,-y\right).
\]
Sometimes we will also write $\mathcal{I}=-I\in\DBGLL$.

\subsection{Partial Fourier transforms}
Given $F\in L^{1}(\DBrdd)$, we use the symbols $\mathcal{F}_{1}$ and $\mathcal{F}_{2}$ to denote the partial
	Fourier transforms 
	$$
	\mathcal{F}_{1}F\left(\xi,y\right)=\widehat{F_{y}}\left(\xi\right)=\int_{\mathbb{R}^{d}}e^{-2\pi i\xi \cdot t}F\left(t,y\right)dt,	 \quad  \,\, \xi,y\in\DBrd
	$$
	$$
	\mathcal{F}_{2}F\left(x,\omega\right)=\widehat{F_{x}}\left(\omega\right)=\int_{\mathbb{R}^{d}}e^{-2\pi i\omega \cdot t}F\left(x,t\right)dt, \quad \,\, x,\omega\in\DBrd,
	$$
	where $\widehat{\cdot}$ denotes the Fourier transform on $L^{1}(\DBrd)$
	and
	\[
	F_{x}\left(y\right)=F\left(x,y\right),\qquad F_{y}\left(x\right)=F\left(x,y\right)\quad \,\, x,y\in\DBrd
	\]
	are the sections of $F$ at fixed $x$ and $y$,	respectively. The definition is well-posed thanks to Fubini's theorem, which implies that $F_{x}\in L^{1}\left(\mathbb{R}_{y}^{d}\right)$
for a.e. $x\in\mathbb{R}^{d}$ and $F_{y}\in L^{1}\left(\mathbb{R}_{x}^{d}\right)$
for a.e. $y\in\mathbb{R}^{d}$, thus $\mathcal{F}_{1}F$ and $\mathcal{F}_{2}F$
are indeed well defined. The Fourier transform $\mathcal{F}$  is therefore related to the partial Fourier transforms as 
\[
\mathcal{F}=\mathcal{F}_{1}\mathcal{F}_{2}=\mathcal{F}_{2}\mathcal{F}_{1}.
\]
Using Plancherel's theorem and properties of modulation spaces (Proposition \ref{DBprop M1}, item $(iv)$), the following extension of the partial Fourier transform is routine.
\begin{lemma} \label{DBpartial Fou isom}
	\begin{enumerate}[label=(\roman*)]
	\item The partial Fourier transform $\mathcal{F}_{2}$ is a unitary operator on $L^{2}(\DBrdd)$.
	In particular, 
	\[
	\mathcal{F}_{2}^{*}F\left(x,y\right)=\mathcal{F}_{2}^{-1}F\left(x,y\right)=\mathcal{F}_{2}F\left(x,-y\right)=\mathfrak{T}_{\mathcal{I}_{2}}\mathcal{F}_{2}F\left(x,y\right),
	\]
	where $\mathcal{I}_{2}=\left(\begin{array}{cc}
	I & 0\\
	0 & -I
	\end{array}\right)$.\\
	\item The partial Fourier transform $\mathcal{F}_{2}$ is an isomorphism on $M^1(\DBrdd)$ and on $M^{\infty}(\DBrdd)$.  
	\end{enumerate}
\end{lemma}

%
%
\section{Matrix-Wigner distributions}
Let us define the main characters of this survey. 
\begin{definition}
	\label{DBbiltfr}Let $A=\left(\begin{array}{cc}
	A_{11} & A_{12}\\
	A_{21} & A_{22}
	\end{array}\right)\in\DBGLL.$ The time-frequency distribution\textbf{ }of Wigner type for $f$
	and $g$ associated with $A$ (in short: matrix-Wigner distribution,
	MWD) is defined for suitable functions $f,g$ as
\begin{equation}\label{DBBAi}
\mathcal{B}_{A}\left(f,g\right)\left(x,\omega\right)=\mathcal{F}_{2}\mathfrak{T}_{A}\left(f\otimes\overline{g}\right)\left(x,\omega\right).
\end{equation}
	When $g=f$, we write $\mathcal{B}_{A}f$ for $\mathcal{B}_{A}\left(f,f\right)$.
\end{definition}
Explicitly, $\mathcal{B}_A$ is given by
\[ \mathcal{B}_{A}\left(f,g\right)\left(x,\omega\right)=\int_{\mathbb{R}^{d}}e^{-2\pi i\omega \cdot y}f\left(A_{11}x+A_{12}y\right)\overline{g\left(A_{21}x+A_{22}y\right)}dy. \]

This definition is meaningful on many function spaces. A first result is for the triple $(M^1,L^2,M^{\infty})$.
\begin{proposition}
	\label{DBdef bilA triple} Assume $A\in\DBGLL$. 
	\begin{enumerate}[label=(\roman*)]
		\item If $f,g\in L^{2}(\mathbb{R}^{d})$, then $\mathcal{B}_{A}(f,g)\in L^{2}(\DBrdd)$
		and the mapping $\mathcal{B}_{A}:L^{2}(\mathbb{R}^{d})\times L^{2}(\mathbb{R}^{d})\rightarrow L^{2}(\DBrdd)$
		is continuous. Furthermore, $\mathrm{span}\left\{ \mathcal{B}_{A}(f,g)\,|\,f,g\in L^{2}(\mathbb{R}^{d})\right\} $
		is a dense subset of $L^{2}(\DBrdd)$.
		\item If $f,g\in M^1(\mathbb{R}^{d})$, then $\mathcal{B}_{A}(f,g)\in M^1(\mathbb{R}^{2d})$
		and the mapping  $\mathcal{B}_{A}:M^1(\mathbb{R}^{d})\times M^1(\mathbb{R}^{d})\rightarrow M^1(\DBrdd)$
		is continuous.
		\item If $f,g\in M^{\infty}(\mathbb{R}^{d})$, then $\mathcal{\mathcal{B}}_{A}(f,g) \in M^{\infty}(\DBrdd)$
		and the mapping  $\mathcal{B}_{A}:M^{\infty}(\mathbb{R}^{d})\times M^{\infty}(\mathbb{R}^{d})\rightarrow M^{\infty}(\DBrdd)$
		is continuous. 
	\end{enumerate}
\end{proposition}

%

	
The standard time-frequency representations covered within this framework include for instance:
	\begin{itemize}
		\item the short-time Fourier transform:
\begin{equation}\label{DBASTFT}
V_{g}f\left(x,\omega\right)=\int_{\DBrd} e^{-2\pi i\omega \cdot  y}f\left(y\right)\overline{g\left(y-x\right)}dy=\mathcal{B}_{A_{ST}}\left(f,g\right)\left(x,\omega\right), 
\end{equation} where
\[ A_{ST}=\left(\begin{array}{cc}
0 & I\\
-I & I
\end{array}\right);\]
	\item the  cross-ambiguity function: 
\begin{equation}\label{DBambiguity}
Amb\left(f,g\right)\left(x,\omega\right)=\int_{\DBrd} e^{-2\pi i\omega \cdot  y}f\left(y+\frac{x}{2}\right)\overline{g\left(y-\frac{x}{2}\right)}dy=\mathcal{B}_{A_{Amb}}\left(f,g\right)\left(x,\omega\right),
\end{equation}
where 
\[
A_{Amb}=\left(\begin{array}{cc}
\frac{1}{2}I & I\\
-\frac{1}{2}I & I
\end{array}\right);
\]

	\item the Wigner distribution:
	\begin{equation}\label{DBwig}
	W\left(f,g\right)\left(x,\omega\right)=\int_{\DBrd} e^{-2\pi i\omega \cdot  y}f\left(x+\frac{y}{2}\right)\overline{g\left(x-\frac{y}{2}\right)}dy;
	\end{equation}
	\item the Rihaczek distribution:
	\begin{equation}\label{DBrihaczek}
	R\left(f,g\right)\left(x,\omega\right)=\int_{\DBrd} e^{-2\pi i\omega \cdot  y}f\left(x\right)\overline{g\left(x-y\right)}dy=e^{-2\pi ix\cdot \omega}f\left(x\right)\overline{\hat{g}\left(\omega\right)}.
	\end{equation}

\end{itemize}
The latter two distributions are special cases of the  $\tau$-Wigner distribution defined in \eqref{DBtauwig}. For any $\tau\in\left[0,1\right]$, we have
\begin{equation*}
W_{\tau}\left(f,g\right)\left(x,\omega\right)=\mathcal{B}_{A_{\tau}}\left(f,g\right)\left(x,\omega\right),
\end{equation*}
where 
\begin{equation}\label{DBAtau}
A_{\tau}=\left(\begin{array}{cc}
I & \tau I\\
I & -\left(1-\tau\right)I
\end{array}\right).
\end{equation}

The list of elementary properties is in line of those for the short-time Fourier transform or the Wigner distribution. Mostly the proof is a straightforward computation, and we refer to \cite{DBbayer,DBCT18} for the details. The interesting aspect is how the parametrizing matrix $A$ intervenes in the formulas for $\mathcal{B}_A$. 

\begin{proposition}\label{DBalg prop BA}
	Let $A\in\DBGLL$ and $f,g\in M^1(\DBrd)$. The following properties hold:
		\begin{enumerate}[label=(\roman*)]
		\item \textbf{Interchanging $f$ and $g$}:
	\[
	\mathcal{B}_{A}\left(g,f\right)\left(x,\omega\right)=\overline{\mathcal{B}_{C_1}\left(f,g\right)\left(x,\omega\right)},\quad \left(x,\omega\right)\in\DBrdd,
	\]

where 
\[
C_1=\tilde{I}A\mathcal{I}_{2}=\left(\begin{array}{cc}
0 & I\\
I & 0
\end{array}\right)\left(\begin{array}{cc}
A_{11} & A_{12}\\
A_{21} & A_{22}
\end{array}\right)\left(\begin{array}{cc}
I & 0\\
0 & -I
\end{array}\right)=\left(\begin{array}{cc}
A_{21} & -A_{22}\\
A_{11} & -A_{12}
\end{array}\right).
\]
In particular, $\mathcal{B}_{A}f$ is a real-valued function if and
only if $A=C$, namely 
\[
A_{11}=A_{21},\qquad A_{12}=-A_{22}.
\] 
\item \textbf{Behaviour of Fourier transforms}:  \[
\mathcal{B}_{A}\left(\hat{f},\hat{g}\right)\left(x,\omega\right)=\left|\det A\right|^{-1}\mathcal{B}_{C_2}\left(f,g\right)\left(-\omega,x\right), \quad \left(x,\omega\right)\in\DBrdd,
\]
where 
\[
C_2=\mathcal{I}_{2}A^{\#}\tilde{I}=\left(\begin{array}{cc}
I & 0\\
0 & -I
\end{array}\right)\left(A^{-1}\right)^{\top}\left(\begin{array}{cc}
0 & I\\
I & 0
\end{array}\right).\]
\item \textbf{Fourier transform of a MWD}:
 \begin{equation}\label{DBFouB}
\mathcal{F}\mathcal{B}_{A}\left(f,g\right)\left(\xi,\eta\right)=\mathcal{B}_{AJ}\left(f,g\right)\left(\eta,\xi\right),
\end{equation}
where 
\[
AJ=\left(\begin{array}{cc}
A_{11} & A_{12}\\
A_{21} & A_{22}
\end{array}\right)\left(\begin{array}{cc}
0 & I\\
-I & 0
\end{array}\right)=\left(\begin{array}{cc}
-A_{12} & A_{11}\\
-A_{22} & A_{21}
\end{array}\right).
\]
\end{enumerate}

\end{proposition}

\subsection{Connection to the short-time Fourier transform}
We first investigate the relation of the time-frequency representations $\mathcal{B}_A$ to the ordinary STFT. Whereas the Wigner distribution and the ambiguity transform coincide with the STFT up to normalization, the time-frequency representation $\mathcal{B}_A$ can be written as a STFT only under an  extra condition.

\begin{definition}
	A block matrix $A=\left(\begin{array}{cc}
	A_{11} & A_{12}\\
	A_{21} & A_{22}
	\end{array}\right)\in\mathbb{R}^{2d\times2d}$ is called left-regular (resp. right-regular), if the submatrices $A_{11},A_{21}\in\mathbb{R}^{d\times d}$
	(resp. $A_{12},A_{22}\in\mathbb{R}^{d\times d}$) are invertible. 
\end{definition}

	It is not difficult to prove that $A=\left(\begin{array}{cc}
	A_{11} & A_{12}\\
	A_{21} & A_{22}
	\end{array}\right)\in\DBGLL$ is left-regular (resp. right-regular) if and only if the matrix\footnote{Beware that $(A^{\#})_{ij}\neq A^{\#}_{ij} = (A_{ij}^{\top})^{-1}$, $i,j=1,2$.}  $A^{\#}=\left(A^{-1}\right)^{\top}=\left(\begin{array}{cc}
	(A^{\#})_{11} & (A^{\#})_{12}\\
	(A^{\#})_{21} & (A^{\#})_{22}
	\end{array}\right)$ is right-regular (resp. left-regular).

As a matter of fact, the right-regularity of the matrix $A_{ST}$ in \eqref{DBASTFT} stands out at a first glance and one might guess that this is an essential condition to express $\DBcB_A(f,g)$ as a short-time Fourier transform. In fact, this characterization is very strong, as stated in the subsequent results. 

\begin{theorem}[{\cite[Thm. 1.2.5]{DBbayer}}] 
	\label{DBright-reg rep}Assume that $A\in\DBGLL$
	is right-regular. For every $f,g\in M^1 (\DBrd)$ the following formula holds:
	\begin{equation}
	\mathcal{B}_{A}\left(f,g\right)\left(x,\omega\right)=\left|\det A_{12}\right|^{-1}e^{2\pi i A_{12}^{\#}\omega\cdot A_{11}x}V_{\tilde{g}}f\left(c\left(x\right),d\left(\omega\right)\right),\quad x,\omega\in\DBrd, \label{DBright-reg eq rep}
	\end{equation}
	where 
	\[
	c\left(x\right)=\left(A_{11}-A_{12}A_{22}^{-1}A_{21}\right)x,\qquad d\left(\omega\right)=A_{12}^{\#}\omega,\qquad\tilde{g}\left(t\right)=g\left(A_{22}A_{12}^{-1}t\right).
	\]
\end{theorem}
For the sake of clarity, one might use the following formulation:
\begin{theorem}
	\label{DBright-reg char} Given matrices $M,N,P \in \DBbR^{d\times d}$ and $Q,R \in \mathrm{GL}(\DBbR,d)$, set 	\[
	A=\left(\begin{array}{cc}
	Q^{\#}N^{\top}M & Q^{\#}\\
	R\left(Q^{\#}N^{\top}M-P\right) & RQ^{\#}
	\end{array}\right).
	\] Then $A$ is right-regular and
	\[
\mathcal{B}_{A}\left(f,g\right)\left(x,\omega\right)=\left|\det Q\right|e^{2\pi iMx\cdot N\omega}V_{g\circ R}\left(Px,Q\omega\right),\quad x,\omega\in\DBrd,
	\]
	 for any $f,g \in M^1(\DBrd)$.
\end{theorem}
\begin{proof}
The proof is by computation:
\begin{align*}
& \left|\det Q\right| e^{2\pi iMx\cdot N\omega}V_{g\circ R}\left(Px,Q\omega\right) \\ = & \left|\det Q\right| e^{2\pi iMx\cdot N\omega}\int_{\mathbb{R}^{d}}e^{-2\pi iQ\omega\cdot y}f\left(y\right)\overline{g\left(Ry-RPx\right)}dy\\
= & \left|\det Q\right| \int_{\mathbb{R}^{d}}e^{-2\pi iQ\omega\cdot\left(y-Q^{\#}N^{\top}Mx\right)}f\left(y\right)\overline{g\left(Ry-RPx\right)}dy\\
= & \left|\det Q\right| \int_{\mathbb{R}^{d}}e^{-2\pi iQ\omega\cdot y}f\left(y+Q^{\#}N^{\top}Mx\right)\overline{g\left(Ry+R\left(Q^{\#}N^{\top}M-P\right)x\right)}dy\\
= & \int_{\mathbb{R}^{d}}e^{-2\pi i\omega\cdot y}f\left(Q^{\#}y+Q^{\#}N^{\top}Mx\right)\overline{g\left(RQ^{\#}y+R\left(Q^{\#}N^{\top}M-P\right)x\right)}dy\\
= & \mathcal{B}_{A}\left(f,g\right)\left(x,\omega\right),
\end{align*}
where $A=A_{M,N,P,Q,R}$ is as claimed. 
	
\end{proof}

\begin{remark}
	The peculiar way the blocks of $A$ are combined in $$c(x)=\left(A_{11}-A_{12}A_{22}^{-1}A_{21}\right)x$$ is a well-known construction in linear algebra and is usually called \textit{Schur complement}. The Schur complement comes up many times in  our results, ultimately because of its distinctive role in the inversion of block matrices (cf. for instance \cite[Thm. 2.1]{DBlu shiou}). 
\end{remark}

For distributions associated with right-regular matrices, most results about the short-time Fourier transform can be formulated for $\DBcB_{A}$: for instance, one can easily produce orthogonality formulae or a reconstruction formula for $\DBcB_{A}f$. We will study these issues in more generality in the subsequent sections.

\subsection{Main properties of the transformation $\mathcal{B}_{A}$}
 Having in mind that the entire knowledge on the uncertainty principles could be easily transposed here, we give a qualitative result in the spirit of Benedick's theorem for the Fourier transform - which is based on the corresponding uncertainty principle for the STFT \cite[Thm. 2.4.2]{DBGro 2003 uncert}. 

\begin{theorem}[{\cite[Thm. 1.4.3]{DBbayer}}]
	Let $A\in\DBGLL$ be a right-regular matrix. If the support of $\DBcB_{A}(f,g)$ has finite Lebesgue measure, then necessarily $f\equiv 0$ or $g \equiv 0$. 
\end{theorem}

Next we  characterize the boundedness of $\DBcB_{A}(f,g)$ on Lebesgue spaces - which is a completely established issue for the STFT, cf. \cite{DBbdo lebesgue}.

\begin{proposition}
	\label{DBright-regular continuity}Assume that $A\in\DBGLL$
	is right-regular. For any $1\le p \le \infty$ and $q\ge2$ such that $q'\le p\le q$, $f\in L^{p}(\DBrd)$
	and $g\in L^{p'}(\DBrd)$, we have
	\begin{enumerate}[label=(\roman*)]
		\item $\mathcal{B}_{A}(f,g)\in L^{q}(\DBrdd)$,
		with
		\begin{equation}\label{DBqnorm estimate}
		\left\Vert \mathcal{B}_{A}(f,g)\right\Vert _{L^q}\le\frac{\left\Vert f\right\Vert _{L^p}\left\Vert g\right\Vert _{L^{p'}}}{\left|\det A\right|^{\frac{1}{q}}\left|\det A_{12}\right|^{\frac{1}{p}-\frac{1}{q}}\left|\det A_{22}\right|^{\frac{1}{p'}-\frac{1}{q}}}.
		\end{equation}
		\item If $1 < p < \infty$ then $\mathcal{B}_{A}(f,g)\in C_{0}(\DBrdd)$. In particular, $\mathcal{B}_{A}(f,g)\in L^{\infty}(\DBrdd)$.
	\end{enumerate}
	Furthermore, if $1\le p,q \le \infty$ such that $p<q'$ or $p>q$, the map $\DBcB_{A}(f,g):L^p(\DBrd)\times L^{p'}(\DBrd)\rightarrow L^q(\DBrdd)$ is not continuous. 
\end{proposition} 

\begin{proof}
	We refer to the proof of \cite[Prop. 3.9]{DBCT18} for the details concerning the first part. Item $(ii)$ is a direct application of \cite[Prop. 3.2]{DBbdo lebesgue}. 
\end{proof}

The right-regularity of $A$ is not only a technical condition required for \eqref{DBright-reg eq rep} to hold, but also has unexpected effects on the continuity of $\DBcB_{A}$.

\begin{theorem}[{\cite[Theorem 1.2.9]{DBbayer}}]
	Assume $A\in\DBGLL$ such that $\det A_{22}\ne0$
	but $\det A_{12}=0$. Then there exist $f,g\in L^{2}(\DBrd)$
	such that $\mathcal{B}_{A}\left(f,g\right)$ is not a continuous function
	on $\DBrdd$. 
\end{theorem}

Let us exhibit the orthogonality relations, which extend the Parseval identity to time-frequency distributions. The generalization of the orthogonality relations was one of the main motivations for introducing $\DBcB_A$ in \cite{DBbayer}. 
\begin{theorem}[{\cite[Thm. 1.3.1]{DBbayer}}]
	Let $A\in\DBGLL$ and $f_{1},f_{2},g_{1},g_{2}\in L^2(\DBrd)$.
	Then 
	\begin{equation}\label{DBortrel}
	\left\langle \mathcal{B}_{A}\left(f_{1},g_{1}\right),\mathcal{B}_{A}\left(f_{2},g_{2}\right)\right\rangle _{L^{2}(\DBrdd)}=\frac{1}{\left|\det A\right|}\left\langle f_{1},f_{2}\right\rangle _{L^{2}(\DBrd)}\overline{\left\langle g_{1},g_{2}\right\rangle _{L^{2}(\DBrd)}}.
	\end{equation}
	In particular, 
	\[
	\left\Vert \mathcal{B}_{A}\left(f,g\right)\right\Vert _{L^{2}(\DBrdd)}=\frac{1}{\left|\det A\right|^{1/2}}\left\Vert f\right\Vert _{L^{2}(\DBrd)}\left\Vert g\right\Vert _{L^{2}(\DBrd)}.
	\]
	Thus, the representation  $\mathcal{B}_{A,g}:L^{2}(\DBrd)\ni f\mapsto\mathcal{B}_{A}\left(f,g\right)\in L^{2}(\DBrdd)$
	is a non-trivial constant multiple of an isometry whenever $g\not\equiv0$. 
\end{theorem}
The proof follows directly from the definition in \eqref{DBBAi}, since $\DBcF_2$ is unitary and $\mathfrak{T}_A$ is a multiple of a unitary operator.
\begin{corollary}
	If $\left\{e_{n}\right\}_{n\in\mathbb{N}}$ is an orthonormal basis
	for $L^{2}(\DBrd)$, then \[\left\{ \left|\det A\right|^{1/2}\mathcal{B}_{A}\left(e_{m},e_{n}\right)\,|\,m,n\in\mathbb{N}\right\}\]
	is an orthonormal basis for $L^{2}(\DBrdd)$.
\end{corollary}

While the relevance of the orthogonality relations for signal processing or physics purposes has been sometimes debated \cite{DBcohen review}, they are in fact a useful tool for proving several properties of the time-frequency distributions that satisfy them. In particular, orthogonality relations are the main ingredients of a general procedure for reconstructing a signal from the knowledge of its (cross-)time-frequency distribution with a given window. 

\begin{theorem}[{\cite[Cor. 3.1.7]{DBCT18}}] \label{DBinversion BA}
	Assume $A\in\DBGLL$ and fix $g,\gamma\in L^2 (\DBrd)$
	such that $\left\langle g,\gamma\right\rangle \ne0$. Then, for any
	$f\in L^{2}(\DBrd)$, the following inversion
	formula holds:
	\[
	f=\frac{\left|\det A\right|}{\overline{\left\langle g,\gamma\right\rangle }}\mathcal{B}_{A,\gamma}^{*}\mathcal{B}_{A,g}f,
	\]
	where $\mathcal{B}_{A,\gamma}^{*}:L^{2}(\DBrdd)\rightarrow L^{2}(\DBrd)$ is the adjoint operator of $\DBcB_{A,\gamma}\equiv \DBcB_{A}(\cdot,\gamma)$, defined  as 
	\[ \mathcal{B}_{A,\gamma}^{*}H\left(x\right)=\frac{1}{\left|\det A\right|}\int_{\mathbb{R}^{d}}\mathfrak{T}_{A^{\star}}\mathcal{F}_{2}H\left(x,y\right) \gamma \left(y\right)dy,\] 
	with 
	\[
	A^{\star}=\mathcal{I}_{2}A^{-1}\in\DBGLL.
	\]
\end{theorem}

For right-regular matrices the reconstruction can be made more explicit. 
\begin{proposition}[{\cite[Thm. 1.3.3]{DBbayer}}]
	Let $A\in\DBGLL$ be a right-regular
	matrix, and $g,\gamma\in L^{2}(\DBrd)$ such that
	$\left\langle g,\gamma\right\rangle \ne0$. The following inversion
	formula (to be interpreted as vector-valued integral in $L^{2}(\DBrd)$)
	holds for any $f\in L^{2}(\DBrd)$:
	\[
	f=\frac{1}{\left\langle g,\gamma\right\rangle }\int_{\DBrdd}\mathcal{B}_{A,\gamma}f\left(x,\omega\right)\frac{e^{-2\pi iA_{12}^{\#}\omega\cdot A_{11}x}}{\left|\det A_{12}\right|}M_{d\left(\omega\right)}T_{c\left(x\right)}\tilde{g}dxd\omega,
	\]
	where
	\[
	c\left(x\right)=\left(A_{11}-A_{12}A_{22}^{-1}A_{21}\right)x,\qquad d\left(\omega\right)=A_{12}^{\#}\omega,\qquad\tilde{g}\left(t\right)=g\left(A_{22}A_{12}^{-1}t\right).
	\]
\end{proposition}

Another property that is expected to hold for a time-frequency representation is the covariance under phase-space shifts. 

\begin{theorem}[{\cite[Thm. 1.5.1]{DBbayer}}]
	\label{DBcov formula}Let $A\in\DBGLL$.
	For any $f,g\in M^1(\DBrd)$ and $a,b,\alpha,\beta\in\mathbb{R}^{d}$,
	the following formula holds:
	\begin{align}\label{DBcovform}
	\mathcal{B}_{A}\left(M_{\alpha}T_{a}f,M_{\beta}T_{b}g\right)\left(x,\omega\right)& =e^{2\pi i\sigma \cdot s}M_{\left(\rho,-s\right)}T_{\left(r,\sigma\right)}\mathcal{B}_{A}\left(f,g\right)\left(x,\omega\right) \\
	& =e^{2\pi i\sigma \cdot s}e^{2\pi i\left(x\cdot \rho-\omega \cdot s\right)}\mathcal{B}_{A}\left(f,g\right)\left(x-r,\omega-\sigma\right),
	\end{align}
	
	where 
	\[
	\left(\begin{array}{c}
	r\\
	s
	\end{array}\right)=A^{-1}\left(\begin{array}{c}
	a\\
	b
	\end{array}\right),\qquad\left(\begin{array}{c}
	\rho\\
	\sigma
	\end{array}\right)=A^{\top}\left(\begin{array}{c}
	\alpha\\
	-\beta
	\end{array}\right).
	\]
\end{theorem}

Of course, this result encompasses the covariance formula for the $\tau$-Wigner distribution with $A=A_{\tau}$ as in \eqref{DBAtau}, cf. \cite[Prop. 3.3]{DBcnt18} and also for the STFT with $A=A_{ST}$, cf \cite[Lem. 3.1.3]{DBGrochenig_2001_Foundations}.

We now cite an amazing representation result for the STFT of
a MWD, sometimes called the \textit{magic formula} for other distributions (cf. \cite{DBGro ped}). This relation allows to painlessly extend our results to general function spaces tailored for the purposes of time-frequency analysis, namely modulation spaces. 

\begin{theorem}[{\cite[Thm. 1.7.1]{DBbayer}}]
	\label{DBmagic formula}Assume $A\in\DBGLL$
	and $f,g,\psi,\phi\in M^1(\DBrd)$, and set
	$z=\left(z_{1},z_{2}\right)$, $\zeta=\left(\zeta_{1},\zeta_{2}\right)\in\DBrdd$.
	Then, 
	\begin{equation}\label{DBSTP}
	V_{\mathcal{B}_{A}\left(\phi,\psi\right)}\mathcal{B}_{A}\left(f,g\right)\left(z,\zeta\right)=e^{-2\pi iz_{2}\cdot \zeta_{2}}V_{\phi}f\left(a,\alpha\right)\overline{V_{\psi}g\left(b,\beta\right)},
	\end{equation}
	where 
	\[
	\left(\begin{array}{c}
	a\\
	b
	\end{array}\right)=A\mathcal{I}_{2}\left(\begin{array}{c}
	z_{1}\\
	\zeta_{2}
	\end{array}\right)=\left(\begin{array}{c}
	A_{11}z_{1}-A_{12}\zeta_{2}\\
	A_{21}z_{1}-A_{22}\zeta_{2}
	\end{array}\right),
	\]
	\[
	\left(\begin{array}{c}
	\alpha\\
	\beta
	\end{array}\right)=\mathcal{I}_{2}A^{\#}\left(\begin{array}{c}
	\zeta_{1}\\
	z_{2}
	\end{array}\right)=\left(\begin{array}{c}
	(A^{\#})_{11}\zeta_{1}+(A^{\#})_{12}z_{2}\\
	-(A^{\#})_{21}\zeta_{1}-(A^{\#})_{22}z_{2}
	\end{array}\right).
	\]
\end{theorem}

As a concluding remark, we want to underline that the benefits of linear algebra should be appreciated in view of the very short and simple proofs. This aspect should not be underestimated:  the proof of similar results for certain special members has lead to quite cumbersome computations (cf. the proofs for the $\tau$-Wigner distributions in \cite{DBcnt18}). 

\subsection{Cohen class members as perturbations of the Wigner transform}
We already described the heuristics behind the Cohen class of distributions in the introduction. 

\begin{definition}[{\cite{DBGrochenig_2001_Foundations}}]
	A time-frequency distribution $Q$ belongs to the Cohen's class if
	there exists a tempered distribution $\theta\in\mathcal{S}'(\DBrdd)$
	such that
	\[
	Q\left(f,g\right)=W\left(f,g\right)*\theta,\qquad\forall f,g\in\mathcal{S}(\DBrd).
	\]
\end{definition}

Although the Wigner distribution was the main inspiration for the MWDs studied so far, the connection to the Cohen class is by no means clear. This question is the point of departure of the paper \cite{DBCT18} and the following result completely characterizes the intersection between these families. We rephrase {\cite[Thm. 1.1]{DBCT18}} using the Feichtinger algebra as follows.
\begin{theorem}
	\label{DBmaint}
	Let $A\in \DBGLL$. The distribution $\mathcal{B}_{A}$ belongs to the Cohen class if and only if \[A= A_{M}=\left(\begin{array}{cc}
	I\quad & M+(1/2)I\\
	I\quad & M-(1/2)I
	\end{array}\right)  \] as in \eqref{DBAM},
	for some $M\in\mathbb{R}^{d\times d}$. Furthermore, in this case we have 
	\begin{equation}\label{DBBamtheta}
W_M(f,g)\equiv	\mathcal{B}_{A_{M}}\left(f,g\right) =W\left(f,g\right)*\theta_{M},\quad f,g\in\DBS0,
	\end{equation}
	where the Cohen's kernel $\theta_{M} \in {S}'_0(\DBrd)$ is 
	\begin{equation}\label{DBthetaM}
	\theta_{M}=\DBcF \Theta_M,\quad \mbox{with}\quad \Theta_{M}(\xi,\eta)=e^{-2\pi i \xi\cdot M\eta},\quad (\xi,\eta)\in \DBrdd.
	\end{equation}
	If $M$ is invertible, the kernel $\theta_M$ is explicitly 
	\begin{equation}\label{DBthM}
	\theta_{M}\left(x,\omega\right)=\frac{1}{\left|\det M\right|}e^{2\pi ix\cdot M^{-1}\omega},\quad (x,\omega)\in \DBrdd.
	\end{equation}
	We say that $A=A_{M}$ is a {\it Cohen-type matrix associated with $M\in\mathbb{R}^{d\times d}$}.
\end{theorem}

\begin{remark} We mention that, according to the proof of the necessity part in the previous result, a Cohen-type matrix $A$ should be defined by the following conditions on the blocks:  
\begin{equation}\label{DBcohen param}
A_{11}=A_{21}=I, \quad A_{12}-A_{22}=I.
\end{equation}
The choice $A_{22}=M-(1/2)I$ with $M\in \DBbR^{d\times d}$ is thus a suitable parametrization, but by no means the only possible one - and in fact neither the most natural one. The reason underlying our choice appears if one writes down the explicit formula for $\DBcB_{A_M}$ as
\begin{equation}\label{DBBAMp}
W_M(f,g)(x,\omega)=\int_{\DBrd}e^{-2\pi i\omega \cdot  y}f\left(x+\left(M+\frac{1}{2}I\right)y\right)\overline{g\left(x+\left(M-\frac{1}{2}I\right) y\right)}dy,
\end{equation} which reveals the similarity with the Wigner distribution. 
A sort of symmetry with respect to the Wigner distribution (corresponding to $M=0$) immediately stands out. We interpret these representations as a family of ``linear perturbations'' of the Wigner distribution and $M$ as the control parameter, exactly as $\tau$ controls the degree of deviation of $\tau$-Wigner distributions. For this reason, we will refer to $A=A_M$ as the \textit{perturbative form} of a Cohen-type matrix.
The analogy with the $\tau$-Wigner distributions naturally leads to another representation, hence another choice of $A_{22}$ in \eqref{DBcohen param}. A closer inspection of the kernel \eqref{DBkerneltau} and also of \eqref{DBtauwig} reveals that the role of perturbation parameter is not played by $\tau$, rather by the deviation $\mu= \tau-1/2$. In this analogy one chooses $A_{21}=T\in \DBbR^{d\times d}$ and $A_{22}=-(I-T)$ and obtains
\begin{equation}\label{DBBAMa}
W_T(f,g)(x,\omega) \equiv \DBcB_{A_T}(f,g)(x,\omega)=\int_{\DBrd}e^{-2\pi i\omega \cdot y}f\left(x+Ty\right)\overline{g\left(x- (I-T)y\right)}dy,
\end{equation}
which should be compared to \eqref{DBtauwig} (see also \eqref{DBAtau}). Occasionally, we refer to $A_T$ as the \textit{affine form} of the Cohen-type matrix $A$. 
It is clear that the two forms of a Cohen-type matrix are perfectly equivalent, the connection being \begin{equation} \label{DBaffine pert formula}
M=T-(1/2)I.
\end{equation}
\end{remark} 

Therefore, the choice of a form is just a matter of convenience: when studying the properties of $\DBcB_{A}$ as a time-frequency representation, it seems better to explicitly see the effect of the perturbation $M$ (which could be easily turned off setting $M=0$) and use the perturbative form accordingly. As an example of this, the perturbed representation of a Gaussian signal is provided. 

\begin{lemma}[{\cite[Lem. 4.1]{DBCT18}}]
	\label{DBBAf}
	Consider  $A=A_M\in \DBGLL$ as in \eqref{DBAM} and $\varphi_{\lambda}\left(t\right)=e^{-\pi t^{2}/\lambda}$, $\lambda>0$. Then,
	\begin{multline}\label{DBBAgauss}
	W_M\varphi_{\lambda}\left(x,\omega\right)=\left(2\lambda\right)^{d/2}\det\left(S\right)^{-1/2}e^{-2\pi x^{2}/\lambda}\\ \cdot e^{8\pi\left(M^{\top}x\cdot S^{-1}M^{\top}x\right)/\lambda}e^{8\pi iS^{-1}\omega\cdot M^{\top}x}e^{-2\pi\lambda\omega\cdot S^{-1}\omega},
	\end{multline}
	where $S=I+4M^{\top}M\in\mathbb{R}^{d\times d}$.
\end{lemma} 


\subsubsection{Main properties of the Cohen class}
The properties of a time-frequency distribution belonging to the Cohen class are intimately related to the structure of the Cohen kernel. There is an established list of correspondences between the kernel and the properties, which can be used to deduce the following results. See \cite{DBcohen tfa 95,DBCT18,DBjanssen huds 84,DBjanssen posspread 97}. 

\begin{proposition}\label{DBbam prop}
	Assume that $\mathcal{B}_{A}$ belongs to the Cohen's class. For any $f,g \in \DBS0(\DBrd)$, the following
	properties are satisfied:
	\begin{enumerate}[label=(\roman*)]
		\item \textbf{Correct marginal densities}: 
		\[
		\int_{\mathbb{R}^{d}}\mathcal{B}_{A}f\left(x,\omega\right)d\omega=\left|f\left(x\right)\right|^{2},\qquad\int_{\mathbb{R}^{d}}\mathcal{B}_{A}f\left(x,\omega\right)dx=\left|\hat{f}\left(\omega\right)\right|^{2},\quad x,\omega\in\DBrd.
		\]
		In particular, the energy is preserved:
		\[
		\iint_{\DBrdd}\mathcal{B}_{A}f\left(x,\omega\right)dxd\omega=\left\Vert f\right\Vert _{L^{2}}^{2}.
		\]
		\item \textbf{Moyal's identity}: 
		\[
		\left\langle \mathcal{B}_{A}f,\mathcal{B}_{A}g\right\rangle _{L^{2}(\DBrdd)}=\left|\left\langle f,g\right\rangle \right|^{2}.
		\]
		\item \textbf{Symmetry}: for all $x,\omega\in\mathbb{R}^{d}$, 
		\[
		\mathcal{B}_{A}\left(\mathcal{I}f\right)\left(x,\omega\right)=\mathcal{I}\mathcal{B}_{A}f\left(x,\omega\right)=\mathcal{B}_{A}f\left(-x,-\omega\right),
		\]
		\[
		\mathcal{B}_{A}\left(\overline{f}\right)\left(x,\omega\right)=\overline{\mathcal{I}_{2}\mathcal{B}_{A}f\left(x,\omega\right)}=\overline{\mathcal{B}_{A}\left(x,-\omega\right)}.
		\]
		\item \textbf{Convolution properties}: for all $x,\omega\in\DBrd$,
		\[
		\mathcal{B}_{A}\left(f*g\right)\left(x,\omega\right)=\mathcal{B}_{A}f *_1\mathcal{B}_{A}g,
		\]
		\[
		\mathcal{B}_{A}\left(f\cdot g\right)(x,\omega)=\mathcal{B}_{A}f *_2\mathcal{B}_{A}g.
		\]
		Here $*_1$ (and $*_2$) denotes the convolution with respect to the first  (second) variable.
		\item \textbf{Scaling invariance}: setting $U_{\lambda}f\left(t\right)\coloneqq\left|\lambda\right|^{d/2}f\left(\lambda t\right)$, $\lambda\in\DBbR\setminus\{0\}$, $t\in\DBrd$, 
		\[
		\mathcal{B}_{A}\left(U_{\lambda}f\right)\left(x,\omega\right)=\mathcal{B}_{A}f\left(\lambda x,\lambda^{-1}\omega\right).
		\]
		\item \textbf{Strong support property}\footnote{Let $Qf:\mathbb{R}_{\left(x,\omega\right)}^{2d}\rightarrow\mathbb{C}$
			be a time-frequency distribution associated with the signal $f:\mathbb{R}_{t}^{d}\rightarrow\mathbb{C}$
			in a suitable function space. Recall that $Q$ is said to satisfy the strong support property if
					\[
			f\left(x\right)=0\Leftrightarrow Qf\left(x,\omega\right)=0, \quad \forall \omega \in \DBrd, \qquad \hat{f}\left(\omega\right)=0\Leftrightarrow Qf\left(x,\omega\right)=0, \qquad \forall x\in\mathbb{R}^{d}.
			\]}: the only MWDs in the Cohen class satisfying the strong correct support properties are Rihaczek and conjugate-Rihaczek distributions. 
		\item \textbf{Weak support property}\footnote{With the notation of the previous footnote, we say that $Q$ satisfies the weak support property if, for any signal $f$:
			\[
			\pi_{x}\left(\mathrm{supp}Qf\right)\subset\mathcal{C}\left(\mathrm{supp}f\right), \qquad 		\pi_{\omega}\left(\mathrm{supp}Qf\right)\subset\mathcal{C}\left(\mathrm{supp}\hat{f}\right),
			\] where $\pi_{x}:\mathbb{R}_{\left(x,\omega\right)}^{2d}\rightarrow\mathbb{R}_{x}^{d}$
			and $\pi_{\omega}:\mathbb{R}_{\left(x,\omega\right)}^{2d}\rightarrow\mathbb{R}_{\omega}^{d}$
			are the projections onto the first and second factors ($\mathbb{R}_{\left(x,\omega\right)}^{2d}\simeq\mathbb{R}_{x}^{d}\times\mathbb{R}_{\omega}^{d}$) and $\mathcal{C}\left(E\right)$ is the closed convex hull of $E\subset\mathbb{R}^{d}$.}: the only MWDs in Cohen's class satisfying the weak correct support properties are the $\tau$-Wigner distributions with $\tau\in\left[0,1\right]$. 
	\end{enumerate}
\end{proposition}

We now give a few hints on several aspects of interests for both theoretical problems and applications; extensive discussions on these issues may be found in \cite{DBbayer,DBCT18}. \\

\textbf{Real-valuedness.} In view of Proposition \ref{DBalg prop BA} $(i)$, $\DBcB_{A_0}=W$ (the Wigner distribution) is the only real-valued member of the family $\DBcB_{A_M}$. \\

	\textbf{More on marginal densities.} The marginal densities for a general distribution $\mathcal{B}_{A}$
	can be easily computed. For $f,g\in\DBS0(\DBrd)$,
	\begin{align*}
	\int_{\mathbb{R}^{d}}\mathcal{B}_{A}f\left(x,\omega\right)d\omega&=f\left(A_{11}x\right)\overline{f\left(A_{21}x\right)},\\
	\int_{\mathbb{R}^{d}}\mathcal{B}_{A}f\left(x,\omega\right)dx&=\left|\det A\right|^{-1}\hat{f}\left((A^{\#})_{12}\omega\right)\overline{\hat{f}\left(-(A^{\#})_{22}\omega\right)}.
	\end{align*}
	The correct marginal densities are thus recovered if and only if $A_{11}=A_{21}=I$
	and $(A^{\#})_{12}=-(A^{\#})_{22}=I$. These conditions force both $\left|\det A\right|=1$
	and the block structure of $A$ as that of Cohen's type. This fact provides an equivalent characterization of the distributions $\DBcB_{A}$ belonging to the Cohen class: these are exactly those satisfying the correct marginal densities. \\
	
\textbf{Relation between two distributions.}  Let $A_{1}=A_{M_{1}}$ and $A_{2}=A_{M_{2}}$ be two Cohen-type matrices as in \eqref{DBAM}. The two distributions $W_{M_1}$ and $W_{M_2}$ are connected by a Fourier multiplier as follows (\cite[Lem. 4.2]{DBCT18}). For $f,g\in\DBS0(\DBrd)$,
	\begin{equation} \label{DBM1M2 fou mult}
	\mathcal{F}W_{M_2}\left(f,g\right)\left(\xi,\eta\right)=e^{-2\pi i\xi\cdot\left(M_{2}-M_{1}\right)\eta}\mathcal{F}W_{M_1}\left(f,g\right)\left(\xi,\eta\right).
	\end{equation}
	
	Furthermore, if $M_{2}-M_{1}\in\mathrm{GL}\left(d,\mathbb{R}\right)$,
	\[
	\mathcal{B}_{A_{2}}\left(f,g\right)\left(x,\omega\right)=\frac{1}{\left|\det\left(M_{2}-M_{1}\right)\right|}e^{2\pi ix\cdot\left(M_{2}-M_{1}\right)^{-1}\omega}*\mathcal{B}_{A_{1}}\left(f,g\right)\left(x,\omega\right).
	\]
The proofs follow at once from Theorem \ref{DBmaint}. \\
 
\textbf{Regularity of the Cohen kernel.} Let $A=A_{M}\in\DBGLL$
be a Cohen-type matrix and assume in addition that $M\in\mathrm{GL}\left(d,\mathbb{R}\right)$. The Cohen kernel associated with $\DBcB_{A_M}$ is therefore given by \eqref{DBthM}, and we can study its regularity with respect to the scale of modulation spaces: we have (\cite[Prop. 4.8]{DBCT18})
\[
\theta_{M},\DBcF \theta_M \in M^{1,\infty}(\DBrdd)\cap W\left(\mathcal{F}L^{1},L^{\infty}\right)(\DBrdd).
\]
This result has an interesting counterpart on the regularity of $W_M$ on all modulation spaces, in view of the boundedness of Fourier multipliers with symbols in $W\left(\mathcal{F}L^{1},L^{\infty}\right)$ (cf. \cite[Lem. 8]{DBbenyi}):

\begin{theorem}[{\cite[Thm. 4.10]{DBCT18}}]
	Let $A=A_{M}\in\DBGLL$ be a Cohen-type matrix with $M\in\mathrm{GL}\left(d,\mathbb{R}\right)$ and $f\in M^{\infty}(\DBrd)$
	be a signal. Then, for any $1\leq p,q\leq\infty$, we have 
	\[
	Wf\in M^{p,q}(\DBrdd) \Longleftrightarrow W_M f\in M^{p,q}(\DBrdd).
	\]	
\end{theorem}

Unfortunately, one cannot go too far if the non-singularity of $M$ is dropped; as a trivial instance, notice that for $M=0$ one has $\theta_{0}=\delta$, and it is easy to verify that $\delta \in M^{1,\infty}(\DBrdd)\backslash W\left(\mathcal{F}L^{1},L^{\infty}\right)(\DBrdd)$, cf. \cite{DBcdgn tfa bj}. \\ 
	
\textbf{Perturbation and interferences.} The emergence of unwanted artefacts is a well-known drawback of any quadratic representation. The signal processing literature is full of strategies to mitigate these effects (see for instance \cite{DBcohen tfa 95,DBhlaw book,DBhlaw qtf}). For what concerns the Cohen class, it is folklore that the severity of interferences is somewhat related to the decay of the Cohen kernel. In fact, a precise formulation of this principle is rather elusive and recent contributions unravelled further non-trivial fine points (\cite[Prop. 4.4 and Thm. 4.6]{DBcdgdn sympcov interf}). We remark that the chirp-like kernel $\Theta_M=\DBcF{\theta_M}$ does not decay at all, and thus no smoothing effect should be expected for the perturbed representations. This is confirmed by the experiments in dimension $d=1$ in \cite{DBbco quadratic,DBCT18}. The only effect of the perturbation consists of a distortion and relocation of interferences, but there is no damping. Following the engineering literature, we suggest that convolution with suitable decaying distributions may provide some improvement, probably at the price of loosing other nice properties. \\


\textbf{Covariance formula.} For any $z=\left(z_{1},z_{2}\right),\,w=\left(w_{1},w_{2}\right)\in\DBrdd$, the covariance formula \eqref{DBcovform} now reads

\begin{multline}
W_M\left(\pi\left(z\right)f,\pi\left(w\right)g\right)\left(x,\omega\right)=e^{2\pi i\left[\frac{1}{2}\left(z_{2}+w_{2}\right)+M\left(z_{2}-w_{2}\right)\right]\cdot \left(z_{1}-w_{1}\right)} \\ \times M_{J\left(z-w\right)}T_{\mathcal{T}_{M}\left(z,w\right)}W_M\left(f,g\right)\left(x,\omega\right),\label{DBeq:covar cohen}
\end{multline}
where 
\begin{align*}
\mathcal{T}_{M}\left(z,w\right) & =\left(\begin{array}{c}
(1/2)\left(z_{1}+w_{1}\right)+M\left(w_{1}-z_{1}\right)\\
(1/2)\left(z_{2}+w_{2}\right)+M\left(z_{2}-w_{2}\right)
\end{array}\right) \\ & =\frac{1}{2}\left(z+w\right)+\left(\begin{array}{cc}
-M & 0\\
0 & M 
\end{array}\right) \left(z-w\right).
\end{align*}

Alternatively, adopting the affine representation of $A$ \eqref{DBaffine pert formula}:
\begin{equation}\label{DBPT} 
P_{T}=\left(\begin{array}{cc}
-T& 0\\
0 & -(I-T)
\end{array}\right),\quad I+P_{T}=\left(\begin{array}{cc}
I-T & 0\\
0 & T
\end{array}\right),
\end{equation} \\
 we can also write
\begin{equation}\label{DBTM PT}
\mathcal{T}_{T}\left(z,w\right)=\left(\begin{array}{c}
(I-T)z_1 + T w_1\\
Tz_2 + (I-T)w_2
\end{array}\right) = \left(I+P_{T}\right)z-P_{T}w.
\end{equation}

\textbf{Boundedness on modulation  spaces.} We cite here some results on the continuity of the distributions $\DBcB_{A_M}$ on the aforementioned  spaces. For the sake of clarity, we report a simplified, unweighted form of \cite[Thm. 4.12]{DBCT18}.

\begin{theorem}\label{DBsharpbou}
	Let $A=A_{T}\in\mathrm{GL}\left(2d,\mathbb{\mathbb{R}}\right)$ be a
	Cohen-type matrix. Let $1\le p_{i},q_{i},p,q\le\infty$, $i=1,2$, such
	that 
	\begin{equation}
	p_{i},q_{i}\le q,\qquad i=1,2,\label{DBeq:cond1sharp}
	\end{equation}
	and 
	\begin{equation}
	\frac{1}{p_{1}}+\frac{1}{p_{2}}\ge\frac{1}{p}+\frac{1}{q},\qquad\frac{1}{q_{1}}+\frac{1}{q_{2}}\ge\frac{1}{p}+\frac{1}{q}.\label{DBeq:cond2sharp}
	\end{equation}
	
	\begin{enumerate}[label=(\roman*)]
		\item If $f_{1}\in M^{p_{1},q_{1}}(\DBrd)$
		and $f_{2}\in M^{p_{2},q_{2}}(\DBrd)$,
		then $W_T \left(f_{1},f_{2}\right)\in M^{p,q}\left(\mathbb{R}^{2d}\right)$,
		and the following estimate holds:
		\[
		\left\Vert W_T \left(f_{1},f_{2}\right)\right\Vert _{M^{p,q}}\lesssim_{T}\left\Vert f_{1}\right\Vert _{M^{p_{1},q_{1}}}\left\Vert f_{2}\right\Vert _{M^{p_{2},q_{2}}}.
		\] {
			\item Assume further that both $T$ and $I-T$
			are invertible (equivalently: $A_T$ is right-regular, or $P_{T}$ is invertible, cf. \eqref{DBPT}). If $f_{1}\in M^{p_{1},q_{1}}(\DBrd)$
			and $f_{2}\in M^{p_{2},q_{2}}(\DBrd)$,
			then $W_T \left(f_{1},f_{2}\right)\in W\left(\mathcal{F}L^{p},L^{q}\right)\left(\mathbb{R}^{2d}\right)$,
			and the following estimate holds:
			\[
			\left\Vert W_T \left(f_{1},f_{2}\right)\right\Vert _{W\left(\mathcal{F}L^{p},L^{q}\right)}\lesssim_{T}\left(C_{T}\right)^{1/q-1/p}\left\Vert f_{1}\right\Vert _{M^{p_{1},q_{1}}}\left\Vert f_{2}\right\Vert _{M^{p_{2},q_{2}}},
			\]
			where 
			\begin{equation}\label{DBCT}
			C_{T}=\left|\det T \right| \left| \det\left(I-T\right)\right|>0.
			\end{equation}}
	\end{enumerate}
\end{theorem} 

Sharp estimates and continuity results of this type have been given by some of the authors for the case of $\tau$-Wigner distributions in \cite[Lem.  3.1]{DBcdet18} and \cite{DBcnt18}. 

To conclude this section we remark that one can specialize Proposition
\ref{DBright-regular continuity} in order to characterize the boundedness on Lebesgue spaces at the price of assuming right-regularity of $A_M$, see \cite[Thm. 4.14]{DBCT18}.

\section{Pseudodifferential operators}
In this section we discuss the formalism of pseudodifferential operators that is associated with every time-frequency representation $\DBcB_A$. Imitating the time-frequency analysis of Weyl pseudodifferential
operators, we introduce the following general  calculus for  pseudodifferential operators. 
\begin{theorem}[{\cite[Prop. 2.2.1]{DBbayer}}]
	\label{DBthm:def psdoA}Let $A\in\text{GL}\left(2d,\mathbb{R}\right)$
	and $\sigma\in M^{\infty}\left(\mathbb{R}^{2d}\right)$. The mapping
	$\mathrm{op}_{A}(\sigma)\equiv\sigma^{A}$ defined by duality as
	\[
	\left\langle \sigma^{A}f,g\right\rangle \equiv \left\langle \sigma,\mathcal{B}_{A}\left(g,f\right)\right\rangle ,\qquad f,g\in M^1(\DBrd)
	\]
	is a well-defined linear continuous map from $M^1(\DBrd)$
	to $M^{\infty}(\DBrd)$.
\end{theorem}
The proof easily follows from the continuity of the distribution $\mathcal{B}_{A}: \DBS0(\DBrd)\times\DBS0(\DBrd)\to \DBS0(\DBrdd)$, from Proposition \ref{DBdef bilA triple}.

%
%
%
\begin{definition}
	Let $A\in\text{GL}\left(2d,\mathbb{R}\right)$ and $\sigma\in M^{\infty}\left(\mathbb{R}^{2d}\right)$.
	The mapping defined in Theorem \ref{DBthm:def psdoA}, namely 
	\[
	\sigma^{A}:M^1(\DBrd)\ni f\mapsto\sigma^{A}f\in M^{\infty}(\DBrd):\left\langle \sigma^{A}f,g\right\rangle =\left\langle \sigma,\mathcal{B}_{A}\left(g,f\right)\right\rangle, \quad\forall g\in M^1(\DBrd),
	\]
	is called quantization rule with symbol $\sigma$ associated with
	the matrix-Wigner distribution $\mathcal{B}_{A}$or pseudodifferential
	operator with symbol $\sigma$ associated with the matrix-Wigner distribution
	$\mathcal{B}_{A}$.
\end{definition}
Using Feichtinger's kernel theorem (Theorem \ref{DBfei ker thm}), we now provide a number of equivalent representations for $\sigma^{A}f$. 
\begin{theorem}\label{DBequiv reps of ops}
	Let $A \in \DBGLL$. Let $T:M^1(\DBrd)\rightarrow M^{\infty}(\DBrd)$
	be a continuous linear operator. There exist distributions 
	$k, \sigma, F \in M^{\infty}\left(\mathbb{R}^{2d}\right)$ such that $T$
	admits the following representations:
	\begin{enumerate}
		\item as an integral operator with kernel $k$: $\left\langle Tf,g\right\rangle =\left\langle k,g\otimes\overline{f}\right\rangle $
		for any $f,g\in M^1(\DBrd)$;
		\item as pseudodifferential operator with symbol $\sigma$ associated with
		$\mathcal{B}_{A}$: $T=\sigma^{A}$;
		\item as a superposition (in weak sense) of time-frequency shifts (also called \textit{spreading representation}):
		\[
		T=\iint_{\mathbb{R}^{2d}}F\left(x,\omega\right)T_{x}M_{\omega}dxd\omega.
		\]
	\end{enumerate}
	The relations among $k,\sigma,F$ and $A$ are the following:
	\begin{equation}\label{DBrelazioni ker thm}
	\sigma=\left|\det A\right|\mathcal{F}_{2}\mathfrak{T}_{A}k,\qquad F=\mathcal{F}_{2}\mathfrak{T}_{A_{ST}}k.
	\end{equation}
\end{theorem}
\begin{proof}
	The first representation is exactly the claim of kernel theorem.
	Now set $\sigma=\left|\det A\right|\mathcal{F}_{2}\mathfrak{T}_{A}k\in M^{\infty}\left(\mathbb{R}^{2d}\right)$:
	this is a well-defined distribution, since $\mathcal{F}_{2}$ and $\mathfrak{T}_{A}$
	are isomorphisms on $M^{\infty}\left(\mathbb{R}^{2d}\right)$.
	In particular, for any $f,g\in M^1(\DBrd)$
	we have
	\begin{align*}
	\left\langle Tf,g\right\rangle  & =\left\langle k,g\otimes\overline{f}\right\rangle \\
	& =\left\langle \left|\det A\right|^{-1}\mathfrak{T}_{A}^{-1}\mathcal{F}_{2}^{-1}\sigma,g\otimes\overline{f}\right\rangle \\
	& =\left\langle \sigma,\mathcal{F}_{2}\mathfrak{T}_{A}\left(g\otimes\overline{f}\right)\right\rangle \\
	& =\left\langle \sigma,\mathcal{B}_{A}\left(g,f\right)\right\rangle \\
	& =\left\langle \sigma^{A}f,g\right\rangle .
	\end{align*}
	This proves that $Tf=\sigma^{A}f$ in $M^{\infty}(\DBrd)$.
	The relation between the kernel representation in $1$ and the spreading representation in $3$ is well-known, e.g. \cite{DBGrochenig_2001_Foundations}. It can also be deduced from item $2$ from the special matrix $A_{ST}=\left(\begin{array}{cc}
	0 & I\\
	-I & I
	\end{array}\right)$ and 
	\[
	\left\langle Tf,g\right\rangle =\left\langle F,V_{f}g\right\rangle =\left\langle F,\mathcal{B}_{A_{ST}}\left(g,f\right)\right\rangle , \qquad f,g\in M^1(\DBrd).
	\]
\end{proof}

\begin{remark} 
Since $k=\left| \det A \right|^{-1}\mathfrak{T}_{A^{-1}} \DBcF_2^{-1}\sigma = \left| \det A \right|^{-1}\mathfrak{T}_{\mathcal{I}_2 A^{-1}} \DBcF_1^{-1}\widehat{\sigma} $, one can formally obtain another representation of the third type with a special spreading function: 
\begin{equation} \sigma^{A}f(x)= \frac{1}{\left| \det A \right|}\int_{\mathbb{R}^{2d}} \widehat{\sigma}(\xi,-(A^{-1})_{21}x-(A^{-1})_{22}y) e^{2\pi i \xi \cdot [(A^{-1})_{11}x+(A^{-1})_{22}y]} f(y) d\xi dy.\end{equation}
Notice that the inverse of a Cohen-type matrix $A=A_T$ has the form 
\[ A_T^{-1} = 
\left(\begin{array}{cc}
-(I-T) & T\\
I & -I
\end{array}\right),
\] thus the previous formula becomes
\begin{equation}\label{DBspread rep} \sigma^{A}f(x)= \int_{\mathbb{R}^{2d}} \widehat{\sigma}(\xi,u) e^{-2\pi i (I-T)u \cdot \xi} T_{-u}M_{\xi}f(x) d\xi du.\end{equation} 
This should be compared with \cite[Eq. 14.14]{DBGrochenig_2001_Foundations} and \cite[Eq. 20]{DBcnt18}.
\end{remark}	

We now study the relations among pseudodifferential operators associated
with MWDs and the corresponding symbols. 
\begin{proposition}
	Let $A,B\in\text{GL}\left(2d,\mathbb{R}\right)$ and $\sigma,\rho\in M^{\infty}\left(\mathbb{R}^{2d}\right)$.
	Then,
	\[
	\sigma^{A}=\rho^{B}\quad\Longleftrightarrow\quad\sigma=\frac{\left|\det A\right|}{\left|\det B\right|}\mathcal{F}_{2}\mathfrak{T}_{B^{-1}A}\mathcal{F}_{2}^{-1}\rho
	\]
\end{proposition}
\begin{proof}
	Assume that $T=\sigma^{A}=\rho^{B}$. According to Theorem \ref{DBequiv reps of ops}, $T$ has a distributional
	kernel $k$ such that 
	\[
	\sigma=\left|\det A\right|\mathcal{F}_{2}\mathfrak{T}_{A}k,\qquad\rho=\left|\det B\right|\mathcal{F}_{2}\mathfrak{T}_{B}k.
	\]
	Therefore, 
	
	\begin{alignat*}{1}
	\sigma & =\left|\det A\right|\mathcal{F}_{2}\mathfrak{T}_{A}k\\
	& =\frac{\left|\det A\right|}{\left|\det B\right|}\mathcal{F}_{2}\mathfrak{T}_{A}\mathfrak{T}_{B}^{-1}\mathcal{F}_{2}^{-1}\rho\\
	& =\frac{\left|\det A\right|}{\left|\det B\right|}\mathcal{F}_{2}\mathfrak{T}_{B^{-1}A}\mathcal{F}_{2}^{-1}\rho.
	\end{alignat*}
	On the other side, if $\sigma=\left|\det A\right|\left|\det B\right|^{-1}\mathcal{F}_{2}\mathfrak{T}_{B^{-1}A}\mathcal{F}_{2}^{-1}\rho,$
	then for any $f,g\in M^1(\DBrd)$
	\begin{alignat*}{1}
	\left\langle \sigma^{A}f,g\right\rangle  & =\left\langle \sigma,\mathcal{F}_{2}\mathfrak{T}_{A}\left(f\otimes\overline{g}\right)\right\rangle \\
	& =\left\langle \left|\det A\right|\left|\det B\right|^{-1}\mathcal{F}_{2}\mathfrak{T}_{A}\mathfrak{T}_{B}^{-1}\mathcal{F}_{2}^{-1}\rho,\mathcal{F}_{2}\mathfrak{T}_{A}\left(f\otimes\overline{g}\right)\right\rangle \\
	& =\left\langle \rho,\mathcal{F}_{2}\mathfrak{T}_{B}\left(f\otimes\overline{g}\right)\right\rangle \\
	& =\left\langle \rho^{B}f,g\right\rangle .
	\end{alignat*}
\end{proof}
When the operators are associated with Cohen-type matrices, we have
a more explicit relation that covers the usual rule for $\tau$-Shubin
operators (cf. \cite[Rem. 1.5]{DBtoft cont 1}). The proof is a straightforward application of \eqref{DBM1M2 fou mult}. 

\begin{proposition}
	Let $A_{1}=A_{T_{1}}$, $A_{2}=A_{T_{2}}$ be Cohen-type invertible
	matrices, and $\sigma,\rho\in M^{\infty}\left(\mathbb{R}^{2d}\right)$.
	Then,
	\[
	\sigma_{1}^{T_{1}}=\sigma_{2}^{T_{2}}\quad\Longleftrightarrow\quad\widehat{\sigma_{2}}\left(\xi,\eta\right)=e^{-2\pi i\xi\cdot\left(T_{2}-T_{1}\right)\eta}\widehat{\sigma_{1}}\left(\xi,\eta\right).
	\]
\end{proposition}
It is also interesting to characterize the matrices yielding self-adjoint
operators. 
\begin{proposition} [{\cite[Prop. 2.2.3]{DBbayer}}]
	
	Let $A\in\text{GL}\left(2d,\mathbb{R}\right)$ and $\sigma\in M^{\infty}\left(\mathbb{R}^{2d}\right)$.
	Then 
	\[
	\left(\sigma^{A}\right)^{*}=\rho^{B},
	\]
	where 
	\[
	\rho=\overline{\sigma},\qquad B=\tilde{I}A\mathcal{I}_{2}=\left(\begin{array}{cc}
	A_{21} & -A_{22}\\
	A_{11} & -A_{12}
	\end{array}\right).
	\]
	In particular, $\sigma^{A}$ is self-adjoint if and only if $\sigma=\overline{\sigma}$
	(real symbol) and $B=A$, hence 
	\[
	A_{21}=A_{11},\qquad A_{12}=-A_{22}.
	\]
\end{proposition}
\begin{remark}
	Thus, only matrices of the form
	$\left(\begin{array}{cc}
	P & Q\\
	P & -Q
	\end{array}\right)$, with $P,Q\in\text{GL}\left(d,\mathbb{R}\right)$, give rise to pseudodifferential
	operators which are self-adjoint for real symbols. This occurs for
	Weyl calculus but not for Kohn-Nirenberg operators ($T=0$).
\end{remark}

\subsection{Boundedness results}

\subsubsection{Operators on Lebesgue spaces}
The boundedness of a pseudodifferential operator $\sigma^A$ associated with $\DBcB_{A}$ is intimately related to the boundedness of the distribution $\DBcB_{A}$ on certain function spaces, in view of the duality in the definition of $\sigma^A$. Let us start this section with two easy results. 
\begin{proposition}[{\cite[Theorems 2.2.6, 2.2.7 and 2.2.9]{DBbayer}}] \label{DBcont L1 L2}
	Let $A\in\text{GL}\left(2d,\mathbb{R}\right)$ and $\sigma\in M^{\infty}\left(\mathbb{R}^{2d}\right)$. 
	\begin{enumerate}
		\item If $A$ is right-regular and $\sigma\in L^{1}\left(\mathbb{R}^{2d}\right)$,
		then $\sigma^{A}$ is a bounded operator on $L^{2}(\DBrd)$
		such that
		\[
		\left\Vert \sigma^{A}\right\Vert _{L^2 \to L^2}\le\frac{\left\Vert \sigma\right\Vert _{L^{1}}}{\left|\det A_{12}\right|^{1/2}\left|\det A_{22}\right|^{1/2}}.
		\]
		\item If $\sigma\in L^{2}\left(\mathbb{R}^{2d}\right)$, then $\sigma^{A}$
		is a bounded operator on $L^{2}(\DBrd)$ such
		that
		\[
		\left\Vert \sigma^{A}\right\Vert _{L^2 \to L^2}\le\frac{\left\Vert \sigma\right\Vert _{L^{2}}}{\left|\det A\right|^{1/2}}.
		\]
	\end{enumerate}
\end{proposition}
 If $\sigma\in L^{2}\left(\mathbb{R}^{2d}\right)$, then $\sigma^{A}$
is actually a Hilbert-Schmidt operator, and every Hilbert-Schmidt operator $T$ possesses a symbol $\sigma \in L^2(\DBrdd)$ such that $T=\sigma^A$.

The study on Lebesgue spaces can be largely expanded thanks
to the following result. 
\begin{theorem} \label{DBcont op Lp}
	Let $A\in\text{GL}\left(2d,\mathbb{R}\right)$ be right-regular and
	$\sigma\in L^{q}\left(\mathbb{R}^{2d}\right)$. The quantization mapping
	\[
	\sigma\in L^{q}(\mathbb{R}^{2d})\mapsto\sigma^{A}\in\mathcal{L}\left(L^{p}(\DBrd)\right)
	\]
	is continuous if and only if $q\le2$ and $q\le p\le q'$, with norm
	estimate 
	\[
	\left\Vert \sigma^{A}\right\Vert _{L^p \to L^p}\le\frac{\left\Vert \sigma\right\Vert _{L^{q}}}{\left|\det A\right|^{\frac{1}{q'}}\left|\det A_{12}\right|^{\frac{1}{p}-\frac{1}{q'}}\left|\det A_{22}\right|^{\frac{1}{p'}-\frac{1}{q'}}}.
	\]
\end{theorem}
\begin{proof}
	Assume $f\in L^{p}(\DBrd)$ and $g\in L^{p'}(\DBrd)$, with $p\neq 1$ nor $p\neq \infty$. 
	Therefore, by \eqref{DBqnorm estimate} (switch $q$ and $q'$) and H\"older inequality:
	\begin{align*}
	\left|\left\langle \sigma^{A}f,g\right\rangle \right| & =\left|\left\langle \sigma,\mathcal{B}_{A}\left(g,f\right)\right\rangle \right|\\
	& \le\left\Vert \sigma\right\Vert _{L^{q}}\left\Vert \mathcal{B}_{A}\left(g,f\right)\right\Vert _{L^{q'}}\\
	& \le\frac{\left\Vert \sigma\right\Vert _{L^{q}}}{\left|\det A\right|^{\frac{1}{q'}}\left|\det A_{12}\right|^{\frac{1}{p}-\frac{1}{q'}}\left|\det A_{22}\right|^{\frac{1}{p'}-\frac{1}{q'}}}\left\Vert f\right\Vert _{L^{p}}\left\Vert g\right\Vert _{L^{p'}}.
	\end{align*}
	The non-continuity is a consequence of \cite[Prop. 3.2]{DBbdo lebesgue}. 
\end{proof}

\begin{remark}
Note that the closed graph theorem implies non-continuity of the quantization map. This  means that there exists a symbol $\sigma \in L^q$ for which the operator $\sigma^A$ is not bounded on $L^p(\DBrdd)$, cf. \cite[Prop. 3.4]{DBbdo lebesgue}.
\end{remark}

One could also study compactness and Schatten class properties for these operators. We confine ourselves to prove a result for symbols in the Feichtinger algebra.
%
\begin{theorem}[{\cite[Thm. 2.2.8]{DBbayer}}]
	Let $A\in\DBGLL$ and $\sigma\in M^1\left(\mathbb{R}^{2d}\right)$.
	The operator $\sigma^{A}\in\mathcal{L}\left(L^{2}(\DBrd)\right)$
	belongs to the trace class $\mathfrak{S}^{1}\left(L^{2}(\DBrd)\right)$,
	and the following estimate holds for the trace class norm:
	\[
	\left\Vert \sigma^{A}\right\Vert _{\mathfrak{S}^{1}}\lesssim \left|\det A\right|^{1/2} \left\Vert \sigma\right\Vert _{M^1}.
	\]
\end{theorem}
\begin{proof} The proof follows the outline of \cite{DBgro pois}. Proposition \ref{DBcont L1 L2} immediately yields $\sigma^{A}\in\mathcal{L}\left(L^{2}(\DBrd)\right)$,
	since $M^{1}\left(\mathbb{R}^{2d}\right)\subseteq L^{1}\left(\mathbb{R}^{2d}\right)\cap L^{2}\left(\mathbb{R}^{2d}\right)$.
	In line with the paradigm of time-frequency analysis of operators
	(cf. \cite[Sec. 14.5]{DBGrochenig_2001_Foundations}), let us decompose the action of $\sigma^{A}$
	into elementary pseudodifferential operators with time-frequency shifts
	of a suitable function as symbols. The inversion formula for the STFT
	allows to write
	\[
	\sigma=\int_{\mathbb{R}^{2d}}\int_{\mathbb{R}^{2d}}V_{\Phi}\sigma\left(z,\zeta\right)M_{\zeta}T_{z}\Phi dzd\zeta,
	\]
	for any window function $\Phi\in M^1\left(\mathbb{R}^{2d}\right)$
	with $\left\Vert \Phi\right\Vert _{L^{2}}=1$. Therefore, for any
	$f,g\in M^1(\DBrd)$:
	\begin{align*}
	\left\langle \sigma^{A}f,g\right\rangle  & =\left\langle \sigma,\mathcal{B}_{A}\left(g,f\right)\right\rangle \\
	& =\int_{\mathbb{R}^{2d}}\int_{\mathbb{R}^{2d}}V_{\Phi}\sigma\left(z,\zeta\right)\left\langle M_{\zeta}T_{z}\Phi,\mathcal{B}_{A}\left(g,f\right)\right\rangle dzd\zeta\\
	& =\int_{\mathbb{R}^{2d}}\int_{\mathbb{R}^{2d}}V_{\Phi}\sigma\left(z,\zeta\right)\left\langle \left(M_{\zeta}T_{z}\Phi\right)^{A}f,g\right\rangle dzd\zeta.
	\end{align*}
	This shows that $\sigma^{A}$ acts (in an operator-valued sense on
	$L^{2}$) as a continuous weighted superposition of elementary operators:
	\[
	\sigma^{A}=\int_{\mathbb{R}^{2d}}\int_{\mathbb{R}^{2d}}V_{\Phi}\sigma\left(z,\zeta\right)\left(M_{\zeta}T_{z}\Phi\right)^{A}dzd\zeta.
	\]
	The action of the building blocks $\left(M_{\zeta}T_{z}\Phi\right)^{A}$
	can be unwrapped by means of the magic formula \eqref{DBSTP} provided
	one takes $\Phi=\mathcal{B}_{A}\varphi$ for some $\varphi\in M^1(\DBrd)$
	with $\left\Vert \varphi\right\Vert _{L^{2}}=\left|\det A\right|^{1/4}$
	(see the orthogonality relations \eqref{DBortrel}):
	\begin{align*}
	\left\langle \left(M_{\zeta}T_{z}\Phi\right)^{A}f,g\right\rangle  & =\left\langle M_{\zeta}T_{z}\Phi,\mathcal{B}_{A}\left(g,f\right)\right\rangle \\
	& =\overline{V_{\mathcal{B}_{A}\varphi}\mathcal{B}_{A}\left(g,f\right)\left(z,\zeta\right)}\\
	& =e^{2\pi iz_{2}\cdot \zeta_{2}}V_{\varphi}f\left(b,\beta\right)\overline{V_{\varphi}g\left(a,\alpha\right)},
	\end{align*}
	where $a,\alpha,b,\beta$ are continuous functions of $z$ and $\zeta$.
	In particular, we have 
	\[
	\left(M_{\zeta}T_{z}\Phi\right)^{A}:L^{2}(\DBrd)\rightarrow L^{2}(\DBrd)\quad:\quad f\mapsto e^{2\pi iz_{2} \cdot \zeta_{2}}\left\langle f,M_{\beta}T_{b}\varphi\right\rangle M_{\alpha}T_{a},
	\]
	hence $\left(M_{\zeta}T_{z}\Phi\right)^{A}$ is a rank-one operator with trace class norm given by $\left\Vert \left(M_{\zeta}T_{z}\Phi\right)^{A}\right\Vert _{\mathfrak{S}^{1}}=\left\Vert \varphi\right\Vert _{L^{2}}^{2}=\left|\det A\right|^{1/2}$,
	independent of $z,\zeta$. To conclude, we reconstruct the operator
	$\sigma^{A}$ and compute its norm by means of the estimates for the
	pieces:
	\[
	\left\Vert \sigma^{A}\right\Vert _{\mathfrak{S}^{1}}\le\int_{\mathbb{R}^{2d}}\int_{\mathbb{R}^{2d}}\left|V_{\Phi}\sigma\left(z,\zeta\right)\right|\left\Vert \left(M_{\zeta}T_{z}\Phi\right)^{A}\right\Vert _{\mathfrak{S}^{1}}dzd\zeta\le C_{A}\left\Vert \sigma\right\Vert _{M^1}.
	\]
\end{proof}

\subsubsection{Operators on modulation spaces}
We now study the boundedness on modulation spaces of pseudodifferential operators associated with Cohen-type representations. 

\begin{theorem}[Symbols in $M^{p,q}$]
	Let $A=A_{T}\in\DBGLL$ be a Cohen-type matrix and consider indices
	$1\le p,p_{1},p_{2},q,q_{1},q_{2}\le\infty$, satisfying the following
	relations:
	\[
	p_{1},p_{2}',q_{1},q_{2}'\le q',
	\]
	and
	\[
	\frac{1}{p_{1}}+\frac{1}{p_{2}'}\ge\frac{1}{p'}+\frac{1}{q'},\qquad\frac{1}{q_{1}}+\frac{1}{q_{2}'}\ge\frac{1}{p'}+\frac{1}{q'}.
	\]
	For any $\sigma\in M^{p,q}\left(\mathbb{R}^{2d}\right)$, the pseudodifferential
	operator $\sigma^{A}$ is bounded from $M^{p_{1},q_{1}}(\DBrd)$
	to $M^{p_{2},q_{2}}(\DBrd)$. In particular, 
	\[
	\left\Vert \sigma^{A}\right\Vert _{M^{p_{1},q_{1}}\rightarrow M^{p_{2},q_{2}}}\lesssim_{A}\left\Vert \sigma\right\Vert _{M^{p,q}}.
	\]
\end{theorem}
\begin{proof}
	Under the given assumptions on the indices, Theorem \ref{DBsharpbou} implies that $\mathcal{B}_{A_{T}}\left(g,f\right)\in M^{p',q'}(\DBrd)$
	for any $f\in M^{p_{1},q_{1}}(\DBrd)$ and $g\in M^{p_{2}',q_{2}'}(\DBrd)$.
	Therefore, by the duality of modulation spaces we obtain
	\begin{align*}
	\left|\left\langle \sigma^{A}f,g\right\rangle \right| & =\left|\left\langle \sigma,\mathcal{B}_{A}\left(g,f\right)\right\rangle \right|\\
	& \leq\left\Vert \sigma\right\Vert _{M^{p,q}}\left\Vert \mathcal{B}_{A_{M}}\left(g,f\right)\right\Vert _{M^{p',q'}}\\
	& \lesssim_{A}\left\Vert \sigma\right\Vert _{M^{p,q}}\left\Vert f\right\Vert _{M^{p_{1},q_{1}}}\left\Vert g\right\Vert _{M^{p_{2},q_{2}}}.
	\end{align*}
\end{proof}

For symbols in  $W\left(\mathcal{F}L^{p},L^{q}\right)$ spaces, we can extend \cite[Thm. 1.1]{DBcdet18} to the matrix setting. 

\begin{theorem}[Symbols in $W\left(\mathcal{F}L^{p},L^{q}\right)$]
	Let $A=A_{T}\in\DBGLL$ a right-regular Cohen-type matrix and consider
	indices $1\le p,q,r_{1},r_{2}\le\infty$, satisfying the following
	relations:
	\[
	q\le p',\qquad r_{1},r_{1}',r_{2},r_{2}'\leq p.
	\]
	For any $\sigma\in W\left(\mathcal{F}L^{p},L^{q}\right)\left(\mathbb{R}^{2d}\right)$,
	the pseudodifferential operator $\sigma^{A}$ is bounded on $M^{r_{1},r_{2}}(\DBrd)$;
	in particular,
	\[
	\left\Vert \sigma^{A}\right\Vert _{M^{r_{1},r_{2}}\rightarrow M^{r_{1},r_{2}}}\lesssim_{A}\left\Vert \sigma\right\Vert _{W\left(\mathcal{F}L^{p},L^{q}\right)}.
	\]
\end{theorem}
\begin{proof}
	We isolate two special cases of Theorem \ref{DBsharpbou}. First, for any $f\in M^{p_{1},p_{2}}(\DBrd)$
	and $g\in M^{p_{1}',p_{2}'}(\DBrd)$ we have
	\[
	\left\Vert W_T \left(g,f\right)\right\Vert _{W\left(\mathcal{F}L^{1},L^{\infty}\right)}\lesssim_{T} \left\Vert f\right\Vert _{M^{p_{1},p_{2}}}\left\Vert g\right\Vert _{M^{p_{1}',p_{2}'}},\qquad 1\leq p_{1},p_{2}\leq\infty.
	\]
	This yields that $\sigma^{A}$ is bounded on $M^{p_{1},p_{2}}(\DBrd)$,
	for any $1\leq p_{1},p_{2}\leq\infty$ and that the symbol $\sigma$ is in $W\left(\mathcal{F}L^{\infty},L^{1}\right)\left(\mathbb{R}^{2d}\right)$, because
	\begin{align*}
	\left|\left\langle \sigma^{A}f,g\right\rangle \right| & =\left|\left\langle \sigma,W_T\left(g,f\right)\right\rangle \right|\\
	& \leq\left\Vert \sigma\right\Vert _{W\left(\mathcal{F}L^{\infty},L^{1}\right)}\left\Vert W_T\left(g,f\right)\right\Vert _{W\left(\mathcal{F}L^{1},L^{\infty}\right)}\\
	& \lesssim_{T} \left\Vert \sigma\right\Vert _{W\left(\mathcal{F}L^{\infty},L^{1}\right)}\left\Vert f\right\Vert _{M^{p_{1},p_{2}}}\left\Vert g\right\Vert _{M^{p_{1}',p_{2}'}}.
	\end{align*}
	The second case requires $f,g\in M^{2}(\DBrd)$, then  
	\[
	\left\Vert W_T \left(g,f\right)\right\Vert _{W\left(\mathcal{F}L^{2},L^{2}\right)}\lesssim_{T}\left\Vert f\right\Vert _{M^{2}}\left\Vert g\right\Vert _{M^{2}}.
	\]
	If $\sigma\in W\left(\mathcal{F}L^{2},L^{2}\right)\left(\mathbb{R}^{2d}\right)=L^2(\DBrdd)$,
	then $\sigma^{A}$ is bounded on $M^{2}(\DBrd)=L^2(\DBrd)$
	by similar arguments: 
	\[
	\left\Vert \sigma^{A}\right\Vert _{M^{2}\rightarrow M^{2}}\lesssim_{T}\left\Vert \sigma\right\Vert _{W\left(\mathcal{F}L^{2},L^{2}\right)}.
	\]
	
	We proceed now by complex interpolation of the continuous mapping
	$\mathrm{op}_{A}$ on modulation  spaces; in particular,
	we are dealing with 
	\[
	\mathrm{op}_{A}\,:\,W\left(\mathcal{F}L^{1},L^{\infty}\right)\left(\mathbb{R}^{2d}\right)\times M^{p_{1},p_{2}}(\DBrd)\rightarrow M^{p_{1},p_{2}}(\DBrd),
	\]
	\[
	\mathrm{op}_{A}\,:\,W\left(\mathcal{F}L^{2},L^{2}\right)\left(\mathbb{R}^{2d}\right)\times M^{2}(\DBrd)\rightarrow M^{2}(\DBrd).
	\]
	For $\theta\in\left[0,1\right]$, we have
	\[
	\left[W\left(\mathcal{F}L^{1},L^{\infty}\right),W\left(\mathcal{F}L^{2},L^{2}\right)\right]_{\theta}=W\left(\mathcal{F}L^{p},L^{p'}\right),\qquad2\le p\le\infty,
	\]
	\[
	\left[M^{p_{1},p_{2}},M^{2,2}\right]_{\theta}=M^{r_{1},r_{2}},
	\]
	with 
	\[
	\frac{1}{r_{i}}=\frac{1-\theta}{p_{i}}+\frac{\theta}{2}=\frac{1-\theta}{p_{i}}+\frac{1}{p},\qquad i=1,2.
	\]
	From these estimates we immediately derive the condition $r_{1},r_{1}',r_{2},r_{2}'\le p$.
	The inclusion relations for modulation spaces allow to extend the result
	to $W\left(\mathcal{F}L^{p},L^{q}\right)\left(\mathbb{R}^{2d}\right)$
	for any $q\le p'$. Finally, exchanging the role of $p$ and $p'$
	allows us to cover any $p\in\left[1,\infty\right]$.  
\end{proof}


\subsection{Symbols in the Sj\"ostrand class}

Another important space of symbols is the Sj\"ostrand class. It has been introduced by Sj\"ostrand \cite{DBsjo} to extend the well-behaved H\"ormander class $S^0_{0,0}$ and later recognized to coincide with the modulation space $M^{\infty,1}(\DBrdd)$. Accordingly, it consists of bounded symbols with low regularity in general, namely temperate distributions $\sigma\in\mathcal{S}'(\mathbb{R}^{2d})$ such that
$$ \int_{\mathbb{R}^{2d}}\sup_{z\in\mathbb{R}^{2d}} |\langle \sigma, \pi(z,\zeta)g\rangle| d\zeta<\infty.
$$ 
Nevertheless, they lead to $L^2$-bounded pseudodifferential operators. In fact, much more is true: the family of Weyl operators with symbols in the Sj\"ostrand class is an inverse-closed Banach *-subalgebra of $\mathcal{L}(L^2)(\mathbb{R}^d)$, in the following sense.

\begin{theorem} ~ \label{DBthm sjo}
	\begin{enumerate}
		\item[$(i)$] (\textbf{Boundedness}) If $\sigma\in M^{\infty,1}\left(\mathbb{R}^{2d}\right)$, then $\mathrm{op_W}(\sigma)$ is a bounded operator on $L^2(\mathbb{R}^{d})$.  
		\item[$(ii)$] (\textbf{Algebra property}) If $\sigma_1, \sigma_2 \in M^{\infty,1}\left(\mathbb{R}^{2d}\right)$ and $\mathrm{op_W}(\rho)=\mathrm{op_W}(\sigma_1)\mathrm{op_W}(\sigma_2)$, then $\rho\in M^{\infty,1}\left(\mathbb{R}^{2d}\right)$.
		\item[$(iii)$] (\textbf{Wiener property}) If $\sigma\in M^{\infty,1}\left(\mathbb{R}^{2d}\right)$ and $\mathrm{op_W}(\sigma)$ is invertible on $L^2(\mathbb{R}^{d})$,  then $\left[\mathrm{op_W}(\sigma)\right]^{-1}=\mathrm{op_W}(\rho)$ for some $\rho\in M^{\infty,1}\left(\mathbb{R}^{2d}\right)$.
	\end{enumerate}
\end{theorem}

These results have been put into the context of time-frequency analysis in  \cite{DBGrochenig_2006_Time} and have been generalized, see \cite{DBcnt18,DBgro rze,DBgro stro}. In this section we extend Theorem \ref{DBthm sjo} with respect to MWDs.

Let us first provide some conditions on the matrices for which the associated pseudodifferential operators with symbols in $M^{\infty,1}(\DBrdd)$ are bounded on modulation spaces. 


\begin{theorem}[{\cite[Thm. 2.3.1]{DBbayer}}] \label{DBbound op sjo}
	Let $\sigma \in M^{\infty,1}(\DBrdd)$ and assume $A\in \DBGLL $ is a left-regular matrix. The pseudodifferential operator $\sigma^{A}$ is bounded on all modulation spaces $M^{p,q}(\DBrd)$, $1\leq p,q \leq \infty$, with 
\begin{equation}\label{DBest op sjo}
\left\Vert \sigma^{A}\right\Vert _{M^{p,q}\rightarrow M^{p,q}}\lesssim_{A} \frac{1}{\left|\det A_{11}\right|^{1/p'}\left|\det A_{21}\right|^{1/p}}\cdot\frac{1}{\left|\det (A_{12})^{\#}\right|^{1/q'}\left|\det (A_{22})^{\#}\right|^{1/q}}\left\Vert \sigma\right\Vert _{M^{\infty,1}}.
\end{equation}
\end{theorem}
\begin{proof}
Let $f,g\in M^1(\DBrd)$ and $\Phi\in M^1\left(\mathbb{R}^{2d}\right)\backslash\left\{ 0\right\} $.
Then, 
\begin{align*}
\left|\left\langle \sigma^{A}f,g\right\rangle \right| & =\left|\left\langle \sigma,\mathcal{B}_{A}\left(g,f\right)\right\rangle \right|\\
& =\left|\left\langle V_{\Phi}\sigma,V_{\Phi}\mathcal{B}_{A}\left(g,f\right)\right\rangle \right|\\
& \le\left\Vert V_{\Phi}\sigma\right\Vert _{L^{\infty,1}}\left\Vert V_{\Phi}\mathcal{B}_{A}\left(g,f\right)\right\Vert _{L^{1,\infty}},
\end{align*}
where in the last line we used H\"older inequality for mixed-norm
Lebesgue spaces. Let us choose for instance $\Phi=\mathcal{B}_{A}\varphi$
where $\varphi\in M^1(\DBrd)$ is the Gaussian
function. We introduce the affine transformations 
\[
P_{\zeta}\left(z_{1},z_{2}\right)=P_{1}z+P_{2}\zeta=\left(\begin{array}{cc}
A_{11} & 0\\
0 & \left(A_{12}\right)^{\#}
\end{array}\right)\left(\begin{array}{c}
z_{1}\\
z_{2}
\end{array}\right)+\left(\begin{array}{cc}
0 & -A_{12}\\
\left(A_{11}\right)^{\#} & 0
\end{array}\right)\left(\begin{array}{c}
\zeta_{1}\\
\zeta_{2}
\end{array}\right),
\]
\[
Q_{\zeta}\left(z_{1},z_{2}\right)=Q_{1}z+Q_{2}\zeta=\left(\begin{array}{cc}
A_{21} & 0\\
0 & -\left(A_{22}\right)^{\#}
\end{array}\right)\left(\begin{array}{c}
z_{1}\\
z_{2}
\end{array}\right)+\left(\begin{array}{cc}
0 & -A_{22}\\
-\left(A_{21}\right)^{\#} & 0
\end{array}\right)\left(\begin{array}{c}
\zeta_{1}\\
\zeta_{2}
\end{array}\right),
\]
in according with the magic formula \eqref{DBSTP}, and using again
H\"older's inequality we get
\begin{align*}
\left\Vert V_{\Phi}\mathcal{B}_{A}\left(g,f\right)\right\Vert _{L^{1,\infty}} & =\sup_{\left(\zeta_{1},\zeta_{2}\right)\in\mathbb{R}^{2d}}\int_{\mathbb{R}^{2d}}\left|V_{\varphi}g\left(P_{\zeta}\left(z_{1},z_{2}\right)\right)\right|\left|V_{\varphi}f\left(Q_{\zeta}\left(z_{1},z_{2}\right)\right)\right|dz_{1}dz_{2}\\
& \leq\sup_{\left(\zeta_{1},\zeta_{2}\right)\in\mathbb{R}^{2d}}\left\Vert \left(V_{\varphi}f\right)\circ Q_{\zeta}\right\Vert _{L_{z}^{p,q}}\left\Vert \left(V_{\varphi}g\right)\circ P_{\zeta}\right\Vert _{L_{z}^{p',q'}}\\
& =\frac{\left\Vert V_{\varphi}f\right\Vert _{L^{p,q}}}{\left|\det A_{21}\right|^{1/p}\left|\det\left(A_{22}\right)^{\#}\right|^{1/q}}\frac{\left\Vert V_{\varphi}g\right\Vert _{L^{p',q'}}}{\left|\det A_{11}\right|^{1/p'}\left|\det\left(A_{12}\right)^{\#}\right|^{1/q'}}\\
& \le C\frac{\left\Vert f\right\Vert _{M^{p,q}}}{\left|\det A_{21}\right|^{1/p}\left|\det\left(A_{22}\right)^{\#}\right|^{1/q}}\frac{\left\Vert g\right\Vert _{M^{p',q'}}}{\left|\det A_{11}\right|^{1/p'}\left|\det\left(A_{12}\right)^{\#}\right|^{1/q'}},
\end{align*}
where the constant $C$ does not depend on $f$, $g$ or $A$. 

On the other hand, 
\begin{equation}\label{DBnonunif est}
\left\Vert V_{\Phi}\sigma\right\Vert _{L^{\infty,1}}\le C_{\Phi}\left\Vert \sigma\right\Vert _{M^{\infty,1}},
\end{equation}
where the constant $C_{\Phi}$ depends on $\Phi=\mathcal{B}_{A}\varphi$,
hence on $A$. We conclude by duality and get the claimed result. 

\end{proof}

\begin{remark}This result broadly generalizes \cite[Thm. 14.5.2]{DBGrochenig_2001_Foundations} and confirms again that the Sj\"ostrand's class is a well-suited symbol class leading to bounded operators on modulation spaces. It is worth to mention that the left-regularity assumption for $A$ covers any Cohen-type matrix $A=A_M$. 
\end{remark}

\begin{remark}
Unfortunately, it is not easy to sharpen the estimate \eqref{DBest op sjo} neither in the case of Cohen-type matrices, the main obstruction being the estimate in \eqref{DBnonunif est}. The usual strategy consists of finding a suitable alternative window function $\Psi \in M^1(\DBrdd)$ and then estimate $\left\Vert V_{\Psi} \mathcal{B}_A \phi \right\Vert_{L^1}$, in order to apply \cite[Lem. 11.3.3]{DBGrochenig_2001_Foundations} and \cite[Prop. 11.1.3(a)]{DBGrochenig_2001_Foundations}. This require cumbersome computations even in the case of $\tau$-distributions with $\Psi$ Gaussian function, but the result is a uniform estimate for any $\tau \in [0,1]$, cf. \cite[Lem. 2.3]{DBcdet18}. 
\end{remark}

We now prove a similar boundedness result on $W(\DBcF L^p, L^q)$ spaces. We require a very special case of the symplectic covariance for the Weyl quantization which is stable under matrix perturbations - cf. \cite[Lem. 5.1]{DBcnt18} for the $\tau$-Wigner case. 

\begin{lemma}
	For any symbol $\sigma\in M^{\infty}\left(\mathbb{R}^{2d}\right)$ and
	any Cohen-type matrix $A=A_{T}\in\DBGLL$:
	\[
	\mathcal{F}\mathrm{op}_{A_{T}}\left(\sigma\right)\mathcal{F}^{-1}=\mathrm{op}_{A_{I-T}}\left(\sigma\circ J^{-1}\right).
	\]
\end{lemma}
\begin{proof}
	We use a formal argument and leave to the reader the discussion on
	the function spaces on which it is well defined. Recall the spreading
	representation of the operator $\mathrm{op}_{A_{T}}$ from \eqref{DBspread rep}
\[
	\mathrm{op}_{A_{T}}f\left(x\right)  =\int_{\mathbb{R}^{2d}}\widehat{\sigma}\left(\xi,u\right)e^{-2\pi i\left(I-T\right)u \cdot \xi}T_{-u}M_{\xi}f\left(x\right)dud\xi.
\]
Since $\DBcF T_{-u}M_{\xi}\DBcF^{-1}=e^{2\pi i u\cdot \xi}T_{\xi}M_u$, we obtain
	\begin{align*}
	\mathcal{F}\mathrm{op}_{A_{T}}\mathcal{F}^{-1} & =\int_{\mathbb{R}^{2d}}\widehat{\sigma}\left(\xi,u\right)e^{-\pi i\left(2T-I\right)u \cdot \xi} T_{\xi}M_u dud\xi\\
	& =\mathrm{op}_{A_{I-T}}\left(\sigma\circ J^{-1}\right).
	\end{align*}
\end{proof}

\begin{remark}A comprehensive account of the symplectic covariance for perturbed representation is out of the scope of this paper. This would require to investigate how the metaplectic group should be modified in order to accommodate the perturbations. As a non-trivial example of this issue, we highlight the contribution of de Gosson in \cite{DBdg bj} for $\tau$-pseudodifferential operators. \end{remark}

\begin{theorem}
	For any Cohen-type matrix $A=A_T \in \DBGLL$ and any symbol $\sigma\in M^{\infty,1}\left(\mathbb{R}^{2d}\right)$,
	the operator $\mathrm{op_A} (\sigma)$ is bounded on $W\left(\mathcal{F}L^{p},L^{q}\right)(\DBrd)$
	with
	$$
	\| \mathrm{op_A} (\sigma)\|_{W\left(\mathcal{F}L^{p},L^{q}\right)\rightarrow W\left(\mathcal{F}L^{p},L^{q}\right)}\leq C_A \|\sigma\|_{M}^{\infty,1},$$
	for a suitable $C_A>0$.
\end{theorem}
\begin{proof} The proof follows the pattern of \cite[Thm. 5.6]{DBcnt18}. Set $\sigma_J = \sigma \circ J$ and consider the following commutative diagram:
	\[
	\xymatrix{M^{p,q}(\DBrd)\ar[r]^{\mathrm{op}_{I-T}(\sigma_{J})} &  M^{p,q}(\DBrd)\ar[d]^{\mathcal{F}}\\
		W\left(\mathcal{F}L^{p},L^{q}\right)(\DBrd)\ar[r]^{\mathrm{op}_{\tau}(\sigma)}\ar[u]_{\mathcal{F}^{-1}} &  W\left(\mathcal{F}L^{p},L^{q}\right)(\DBrd)
	}
	\]
	From the formula $V_G(\sigma_J)(z,\zeta)=V_{G\circ J^{-1}}\sigma(Jz,J\zeta)$, for any suitable window $G$, it follows easily that  $\sigma\in M^{\infty,1}\left(\mathbb{R}^{2d}\right)$ implies
	$\sigma_{J} \in  M^{\infty,1}\left(\mathbb{R}^{2d}\right)$. The operator $\mathrm{op}_{I-T}\left(\sigma_J\right)$ is bounded on $M^{p,q}(\DBrd)$ as a consequence of Theorem \ref{DBbound op sjo} 
	and the claim follows at once thanks to the previous lemma.
\end{proof}

The significance of the Sj\"ostrand class as space of symbols comes in many shapes. In \cite{DBGrochenig_2006_Time} an alternative characterization of the Sj\"ostrand class was given in terms of a quasi-diagonalization property satisfied by the Weyl operators with symbol in $M^{\infty,1}$. Let us briefly recall the main ingredients of this result. First, fix a non-zero window $\varphi\in \DBS0(\DBrd)$ and a lattice $\Lambda = A\mathbb{Z}^{2d}\subseteq \mathbb{R}^{2d}$, where $A\in\mathrm{GL}(2d,\mathbb{R}) $, such that 
$\mathcal{G}\left(\varphi,\Lambda\right)$ is a Gabor frame for $L^{2}(\DBrd)$. The action of pseudodifferential operators on time-frequency shifts is described by the entries of the so-called \textit{channel matrix}, that is $$\langle \mathrm{op_W}(\sigma)\pi(z)\varphi,\pi(w)\varphi\rangle,\qquad z,w\in\mathbb{R}^{2d}, $$ or $$M(\sigma)_{\lambda,\mu}\coloneqq \langle \mathrm{op_W}(\sigma)\pi(\lambda)\varphi,\pi(\mu)\varphi\rangle,\qquad \lambda,\mu\in\Lambda,$$ if we restrict to the discrete lattice $\Lambda$. We say that $\mathrm{op_W}$ is almost diagonalized by the Gabor frame $\mathcal{G}(\varphi,\Lambda)$ if the associated channel matrix exhibits some sort of off-diagonal decay - equivalently, if the time-frequency shifts are almost eigenvectors for $\mathrm{op_W}$. 

 \begin{theorem}[{\cite{DBGrochenig_2006_Time}}] \label{DBthm gro sjo}
	Let $\varphi\in \DBS0(\mathbb{R}^{d})(\DBrd)$ be a non-zero window function such that $\mathcal{G}\left(\varphi,\Lambda\right)$ be a Gabor frame for $L^{2}(\DBrd)$.
	The following properties are equivalent:
	\begin{enumerate}
		\item[$(i)$] $\sigma\in M^{\infty,1}(\mathbb{R}^{2d})$.
		\item[$(ii)$] $\sigma\in M^{\infty}(\mathbb{R}^{2d})$ and there exists
		a function $H\in L^{1}(\mathbb{R}^{2d})$ such
		that
		\[
		\left|\left\langle \mathrm{op_W}\left(\sigma\right)\pi\left(z\right)\varphi,\pi\left(w\right)\varphi\right\rangle \right|\le H\left(w-z\right),\qquad\forall w,z\in\mathbb{R}^{2d}.
		\]
		\item[$(iii)$] $\sigma\in M^{\infty}\left(\mathbb{R}^{2d}\right)$ and there exists
		a sequence $h\in\ell^{1}\left(\Lambda\right)$ such that
		\[
		\left|\left\langle \mathrm{op_W}\left(\sigma\right)\pi\left(\mu\right)\varphi,\pi\left(\lambda\right)\varphi\right\rangle \right|\le h\left(\lambda-\mu\right),\qquad\forall\lambda,\mu\in\Lambda.
		\]
	\end{enumerate}
\end{theorem}

\begin{corollary}\label{DBweylchar}
	Under the hypotheses of the previous theorem, assume that  $T:\DBS0(\DBrd)\rightarrow M^{\infty}(\DBrd)$
	is continuous and satisfies one of the following conditions:
	\begin{itemize}
		\item[$(i)$] $\left|\left\langle T\pi\left(z\right)\varphi,\pi\left(w\right)\varphi\right\rangle \right|\le H\left(w-z\right),\quad\forall w,z\in\mathbb{R}^{2d}$
		for some $H\in L^{1}$. 
		\item[$(ii)$] $\left|\left\langle T\pi\left(\mu\right)\varphi,\pi\left(\lambda\right)\varphi\right\rangle \right|\le h\left(\lambda-\mu\right),\quad\forall\lambda,\mu\in\Lambda$
		for some $h\in\ell^{1}$. 
	\end{itemize}
	Then $T=\mathrm{op_W}\left(\sigma\right)$ for some symbol
	$\sigma\in M^{\infty,1}\left(\mathbb{R}^{2d}\right).$
\end{corollary}

The backbone of this result is the interplay between the entries of the channel matrix of $\mathrm{op_W}$ and the short-time Fourier transform of the symbol. Theorem \ref{DBthm gro sjo} can be extended without difficulty to $\tau$-pseudodifferential operators \cite{DBcnt18}. We now indicate a further generalization to operators associated with Cohen-type matrices.

Assume that $A=A_{M}\in\text{GL}\left(2d,\mathbb{R}\right)$
is a matrix of Cohen's type. The following result easily follows from the covariance formula \eqref{DBeq:covar cohen}. 
\begin{lemma}\label{DBstft gaborm}
	Let $A=A_{T}\in\text{GL}\left(2d,\mathbb{R}\right)$ be a matrix of
	Cohen's type and fix a non-zero window $\varphi\in M^1(\DBrd)$,
	then set $\Phi_{A}=W_T\varphi$. Then, for any $\sigma\in M^{\infty}\left(\mathbb{R}^{2d}\right)$
	\[
	\left|\left\langle \sigma^{A}\pi\left(z\right)\varphi,\pi\left(w\right)\varphi\right\rangle \right|=\left|V_{\Phi_{A}}\sigma\left(\mathcal{T}_{T}\left(w,z\right),J\left(w-z\right)\right)\right|=\left|V_{\Phi_{A}}\sigma\left(x,y\right)\right|,
	\]
	and 
	\[
	\left|V_{\Phi_{A}}\sigma\left(x,y\right)\right|=\left|\left\langle \sigma^{A}\pi\left(z\left(x,y\right)\right)\varphi,\pi\left(w\left(x,y\right)\right)\varphi\right\rangle \right|,
	\]
	for any $x,y,z,w\in\mathbb{R}^{2d}$, where $\mathcal{T}_T$ is defined in \eqref{DBTM PT} and 
	\begin{equation}
	z\left(x,y\right)=x+\left(I+P_T\right)Jy,\qquad w\left(x,y\right)=x+P_{T}Jy.\label{DBeq:form z w in xy}
	\end{equation}
\end{lemma}
\begin{proof}
	We have 
	\begin{align*}
	\left|\left\langle \sigma^{A}\pi\left(z\right)\varphi,\pi\left(w\right)\varphi\right\rangle \right| & =\left|\left\langle \sigma,\mathcal{B}_{A}\left(\pi\left(w\right)\varphi,\pi\left(z\right)\varphi\right)\right\rangle \right|\\
	& =\left|\left\langle \sigma,M_{J\left(w-z\right)}T_{\mathcal{T}_{T}\left(w,z\right)}\mathcal{B}_{A}\varphi\left(x,\omega\right)\right\rangle \right|\\
	& =\left|V_{\Phi_{A}}\sigma\left(\mathcal{T}_{T}\left(w,z\right),J\left(w-z\right)\right)\right|.
	\end{align*}
	
	Now, setting $x=\mathcal{T}_{T}\left(w,z\right)$ and $y=J\left(w-z\right)$
	we immediately get Eq. (\ref{DBeq:form z w in xy}). 
\end{proof}
\begin{theorem}
	Let $\varphi\in M^1 (\DBrd)$ be a non-zero window function and assume that $\Lambda$ is a lattice such that $\mathcal{G}\left(\varphi,\Lambda\right)$ is a Gabor frame for $L^{2}(\DBrd)$. For any Cohen-type matrix  $A=A_{T}\in\text{GL}\left(2d,\mathbb{R}\right)$,
	the following properties are equivalent:
	\begin{enumerate}
		\item[$(i)$] $\sigma\in M^{\infty,1}\left(\mathbb{R}^{2d}\right)$.
		\item[$(ii)$] $\sigma\in M^{\infty}\left(\mathbb{R}^{2d}\right)$ and there exists
		a function $H=H_{T}\in L^{1}(\DBrd)$ such
		that 
		\[
		\left|\left\langle \sigma^{A}\pi\left(z\right)\varphi,\pi\left(w\right)\varphi\right\rangle \right|\le H_{T}\left(w-z\right),\qquad\forall w,z\in\mathbb{R}^{2d}.
		\]
		\item[$(iii)$] $\sigma\in M^{\infty}\left(\mathbb{R}^{2d}\right)$ and there exists
		a sequence $h=h_{T}\in\ell^{1}\left(\Lambda\right)$ such that
		\[
		\left|\left\langle \sigma^{A}\pi\left(\mu\right)\varphi,\pi\left(\lambda\right)\varphi\right\rangle \right|\le h_{T}\left(\lambda-\mu\right),\qquad\forall\lambda,\mu\in\mathbb{R}^{2d}.
		\]
	\end{enumerate}
\end{theorem}
\begin{proof}
The proof faithfully mirrors the one for Weyl operators \cite[Theorem 3.2]{DBGrochenig_2006_Time}. We detail here only the case $(i)\Rightarrow (ii)$, the discrete case for $h_T$ is similar. 
For $\varphi \in M^1(\DBrd)$ the $T$-Wigner distribution $\Phi_T = W_T \varphi$ is in $M^{1}(\DBrdd)$ by Theorem \ref{DBsharpbou}.
This implies that the short-time Fourier transform $V_{\Phi_T} \sigma$  is well-defined for
$\sigma\in M^{\infty,1}\left(\mathbb{R}^{2d}\right)$ (cf. \cite[Theorem 11.3.7]{DBGrochenig_2001_Foundations}). 
The main insight here is that the controlling function $H_{T}\in L^1(\mathbb{R}^{d})$ can be provided by the so-called \emph{grand symbol} associated with $\sigma$, which is $ \tilde{H}_{T} (v)=\sup_{u\in\DBrdd} |V_{\Phi_T} \sigma(u,v)|$. By definition of
$M^{\infty,1}\left(\mathbb{R}^{2d}\right)$, we have $\tilde{H}_{T}\in L^1(\DBrdd)$, so that Lemma \ref{DBstft gaborm} implies
\begin{align*}
\left|\left\langle \sigma^A \pi\left(z\right)\varphi,\pi\left(w\right)\varphi\right\rangle \right|&=\left|{V}_{\Phi_{T}}\sigma\left(\mathcal{T}_{T}\left(z,w\right),J\left(w-z\right)\right)\right|\\
&\leq \sup_{u\in\DBrdd} \left|{V}_{\Phi_{T}}\sigma\left(u,J\left(w-z\right)\right)\right| \\
& = \tilde{H}_{T} \left( J \left( w-z \right) \right) . 
\end{align*}

Setting $H_T =\tilde{H}_{T}\circ J$ yields the claim.
\end{proof}

Let us further discuss the main trick of the proof. While the choice of the grand symbol as controlling function is natural in view of the $M^{\infty,1}$ norm, the effect of the perturbation matrix $T$ is confined to the window function $\Phi_T$ and to the time variable of the short-time Fourier transform of the symbol. A natural question is then the following: what happens if we try to control the time dependence of $V_{\Phi_{\tau}} \sigma$? Following the pattern of \cite[Thm. 4.3]{DBcnt18}, this remark provides a similar characterization for symbols belonging to the modulation  space $W(\DBcF L^{\infty},L^1) = \DBcF M^{\infty,1}$.  

\begin{theorem}
	Let $\varphi\in M^1 (\DBrd)$ be a non-zero window function. For any right-regular Cohen-type matrix  $A=A_{T}\in\text{GL}\left(2d,\mathbb{R}\right)$,
	the following properties are equivalent:
	\begin{enumerate}
		\item[$(i)$] $\sigma\in W(\DBcF L^{\infty},L^1)\left(\mathbb{R}^{2d}\right)$.
		\item[$(ii)$] $\sigma\in M^{\infty}\left(\mathbb{R}^{2d}\right)$ and there exists
		a function $H=H_{T}\in L^{1}(\DBrd)$ such
		that 
		\[
		\left|\left\langle \sigma^{A}\pi\left(z\right)\varphi,\pi\left(w\right)\varphi\right\rangle \right|\le H_{T}(w-\mathcal{U}_T z),\qquad\forall w,z\in\mathbb{R}^{2d},
		\]
		where 
		\[
		\mathcal{U}_{T}=-\left(\begin{array}{cc}
		\left(I-T\right)^{-1}T & 0\\
		0 & T^{-1}\left(I-T\right)
		\end{array}\right).
		\]
		
	\end{enumerate}
\end{theorem}

\begin{proof}
	The proof is a straightforward adjustment of the one provided for \cite[Thm. 4.3]{DBcnt18}. Again, we detail here only the case $(i)\Rightarrow (ii)$ for the purpose of tracking the origin of $\mathcal{U}_T$. If $\phi \in M^1(\DBrd)$, $\phi \neq 0$, then $\Phi_T= W_T\phi \in M^1 = W(\DBcF L^1, L^1)$ by Theorem \ref{DBsharpbou}. For $\sigma\in W\left(\mathcal{F}L^{\infty},L^{1}\right)$, we have that ${V}_{\Phi_{T}}\sigma$ is well defined and
	\[
	\tilde{H}_T\left(x\right)=\sup_{y\in\mathbb{R}^{2d}}\left|{V}_{\Phi_{T}}\sigma\left(x,y\right)\right|\in L^{1}\left(\mathbb{R}^{2d}\right).
	\]
	From Lemma \ref{DBstft gaborm} we infer
	\begin{flalign*}
	\left|\left\langle \mathrm{op}_T \left(\sigma\right)\pi\left(z\right)\varphi,\pi\left(w\right)\varphi\right\rangle \right| & =\left|{V}_{\Phi_{T}}\sigma\left(\mathcal{T}_{t}\left(w,z\right),J\left(w-z\right)\right)\right|\\
	& \le\sup_{y\in\mathbb{R}^{2d}}\left|{V}_{\Phi_{T}}\sigma\left(\mathcal{T}_{T}\left(w,z\right),y\right)\right|\\
	& =\tilde{H}_T\left(\mathcal{T}_{T}\left(w,z\right)\right).
	\end{flalign*}
	Notice that if $A_T$ is right-regular, then $I+P_T$ from \eqref{DBPT} is invertible (and the converse holds, too). In particular, we have
	\[
	(I+P_T)^{-1}\left(\mathcal{T}_{T}\left(w,z\right)\right)=w- (I+P_T)^{-1}P_Tz = w-\mathcal{U}_{\tau}z,
	\]
	and thus $\tilde{H}_{T}\left(\mathcal{T}_{T}\left(w,z\right)\right)=\tilde{H}_{T}\left(\mathcal{B}_{T}^{-1}\left(w-\mathcal{U}_{T}z\right)\right)$. Define  $H_T= \tilde{H}_{T}\circ (I+P_T)^{-1}$; then $H_T\in L^{1}(\DBrdd)$
	since $\left\Vert H_T\right\Vert _{L^{1}}=\Vert \tilde{H}_{T}\circ (I+P_T)^{-1}\Vert _{L^{1}}\asymp\Vert \tilde{H}_{T}\Vert _{{L}^{1}}<\infty$. 
	\end{proof}

\begin{remark} \begin{enumerate}[label=(\roman*)]
		\item The almost diagonalization of the (continuous) channel matrix does not survive the perturbation, but the new result can be interpreted as a measure of the concentration of the time-frequency representation of $\mathrm{op}_{T} (\sigma)$ along the graph of the map $\mathcal{U}_{T}$.
		\item We recover \cite[Thm. 4.3]{DBcnt18}, where the same problem was first studied for $\tau$-operators with $\tau \in (0,1)$. One may push further the analogy and generalize the results for $\tau=0$ or $\tau=1$. 
		\item As already remarked for $\tau$-operators in \cite{DBcnt18}, the discrete characterization via Gabor frames is lost: for a given lattice $\Lambda$, the inclusion $\mathcal{U}_\tau \Lambda \subseteq \Lambda$ exclusively holds for $\tau=1/2$ (the Weyl transform). 
		\item The boundedness and algebraic properties of $A_T$-operators with symbols in $W(\DBcF L^{\infty},L^1)$ proved in \cite{DBcnt18} rely on the characterization as generalized metaplectic operators according to \cite[Def. 1.1]{DBcgnr}. In order to benefit from this framework, it is necessary for $\mathcal{U}_T$ to be a symplectic matrix, the latter condition being realized if and only if $T$ is a symmetric matrix (cf. \cite[(2.4) and (2.5)]{DBdg sympmeth} and notice that $T^{-1}(I-T)=(I-T)T^{-1}$). 
		\end{enumerate}
\end{remark}	
	
\section*{Acknowledgments}	
E. Cordero  and S. I. Trapasso are members of the Gruppo Nazionale per
l'Analisi Matematica, la Probabilit\`{a} e le loro Applicazioni (GNAMPA) of the Istituto Nazionale di Alta Matematica (INdAM). K. Gr\"ochenig acknowledges support from the Austrian Science Fund FWF, project P31887-N32.

\end{document}